\documentclass[12pt]{amsart}

\usepackage{amsmath,amssymb,a4}
\usepackage{amsfonts}
\usepackage{enumerate}
\usepackage[mathscr]{eucal}
\usepackage[usenames]{color}
\usepackage{verbatim}
\usepackage{amsthm}
\usepackage{amscd}
\usepackage{appendix}
\usepackage{tikz}
\usepackage{hyperref}
\usepackage{cleveref}
\usepackage[utf8]{inputenc}
\usepackage{amsxtra,array,mathrsfs}
\usepackage[shortlabels]{enumitem}

\newtheorem{innercustomgeneric}{\customgenericname}
\providecommand{\customgenericname}{}

\newcommand{\newcustomtheorem}[2]{\newenvironment{#1}[1]
  {\renewcommand\customgenericname{#2}
   \renewcommand\theinnercustomgeneric{##1}\innercustomgeneric}{\endinnercustomgeneric}}

\newcustomtheorem{customthm}{Theorem}

\newcommand{\newcustomlemma}[2]{\newenvironment{#1}[1]
  {\renewcommand\customgenericname{#2}
   \renewcommand\theinnercustomgeneric{##1} \innercustomgeneric}{\endinnercustomgeneric}}

\newcustomlemma{customlemma}{Lemma}

\newcustomlemma{customproposition}{Proposition}

\newcustomlemma{customclaim}{Claim}

\theoremstyle{plain}
\newtheorem{theorem}{Theorem}[section]
\newtheorem{lemma}[theorem]{Lemma}
\newtheorem{corollary}[theorem]{Corollary}
\newtheorem{proposition}[theorem]{Proposition}
\theoremstyle{definition}
\newtheorem{remark}[theorem]{Remark}

\newtheorem{definition}[theorem]{Definition}

\numberwithin{equation}{section}



\def\esssup{\operatornamewithlimits{ess\,sup}}
\def\essinf{\operatornamewithlimits{ess\,inf}}

\newcommand{\bbr}{\mathbb{R}}
\newcommand{\bbrn}{\mathbb{R}^n}
\newcommand{\bbz}{\mathbb{Z}}

\newcommand{\qq}{\qquad}

\def\HH{\mathcal H}
\def\UU{\mathcal U}
\def\VV{\mathcal V}

\begin{document}

\title[Weighted Extensions of Stein's Theorem]{Weighted Extensions of Stein's Theorem for Linear and Multilinear Operators}

\author {Mar\'ia J.\ Carro} 
\address{M. J.\ Carro,    Departamento de An\'alisis y Matem\'atica Aplicada. Universidad Complutense de Madrid, Madrid 28280, Spain}
\email{mjcarro@ucm.es}
\author{Bae Jun Park}
\address{B. Park, Department of Mathematics, Sungkyunkwan University, Suwon 16419, Republic of Korea}
\email{bpark43@skku.edu}

\thanks{The first author was partially supported by grants PID2020-113048GB-I00, PID2024-155917NBI00, and CEX-2023-001347-S, funded by MCIN/AEI/ 10.13039/501100011033, and Grupo UCM970966. The second author was supported in part by NRF grant RS-2025-20512969 and by the Open KIAS Center at Korea Institute for Advanced Study. }

\begin{abstract}
We study weighted estimates for linear and multilinear integral operators whose kernels satisfy only size conditions. Extending a theorem of E. Stein and its refinement by Soria and Weiss, we prove weighted estimates on Herz and Ces\`aro type spaces, together with multilinear strong-type and weak-type analogues. As applications, we derive consequences for a range of rough singular integral operators and related variants, including linear, oscillatory, and multilinear settings.
\end{abstract}

\maketitle
\pagestyle{headings}\pagenumbering{arabic}\thispagestyle{plain}

\section{Introduction} 
In 1967, E. Stein \cite{St1967} proved the following result for power weights:  Let
\begin{equation}\label{operatorT}
Tf(x)=\int_{\bbrn}K(x,y)f(y)\,dy
\end{equation}
be an integral operator whose kernel satisfies
\begin{equation}\label{sizecon}
|K(x,y)|\le \frac{B}{|x-y|^n}, \qquad x\neq y,
\end{equation}
for some constant $B>0$. If, for some $1<p_0<\infty$,
\begin{equation}\label{LpboundT}
T:L^{p_0}(\bbrn)\longrightarrow L^{p_0}(\bbrn),
\end{equation}
then
$$T:L^{p_0}\big( |x|^\alpha\big)\longrightarrow L^{p_0}\big(|x|^\alpha\big)$$
for every
$$ -n<\alpha<n(p_0-1).$$
This range coincides precisely with the range for which the power weight $|x|^\alpha$ belongs to the Muckenhoupt class $A_{p_0}$, introduced in \cite{m:m}, whose definition will be recalled in Subsection \ref{subsec:Apweights}.
It is worth pointing out that, when the result in \cite{St1967} appeared, Muckenhoupt weights had not yet been introduced in the literature.
It is therefore natural to ask whether the same conclusion holds for every $w\in A_{p_0}$. 
However, such a statement would be much stronger than Stein's original statement, 
since, if such an estimate were true for every $w\in A_{p_0}$, then Rubio de Francia's extrapolation would imply that $T$ is bounded on $L^q(\bbrn)$ for every $q>1$. Thus the size estimate alone should not be expected to yield the full $A_{p_0}$-weighted theory without additional regularity or structural assumptions on the kernel.   
In 1994, F. Soria and  G. Weiss \cite{So_We1994} showed, however, that the conclusion remains valid for a distinguished subclass of $A_p$ weights satisfying a dyadic annular comparability condition.
To state the result of Soria and Weiss, for $k\in\bbz$, set
\begin{equation*}
I_k:=\big\{x\in\bbrn:2^{k-1}\le |x|<2^k\big\}, \quad I_k^*:=\big\{x\in\bbrn:2^{k-2}\le |x|<2^{k+1}\big\}.
\end{equation*}

\begin{customthm}{A}[\cite{So_We1994}]\label{SoWetheorem}
Let $T$ be given by \eqref{operatorT}, and assume that \eqref{sizecon} holds.
Let $w$ be a weight satisfying
\begin{equation}\label{consondyad}
\esssup_{I_k} w \le C\,\essinf_{I_k^*} w, \qquad k\in\bbz.
\end{equation}
Then the following assertions hold.
\begin{enumerate}
\item  If $1<p_0<\infty$, $w\in A_{p_0}$, and \eqref{LpboundT} holds, then
$$
T:L^{p_0}(w)\longrightarrow L^{p_0}(w).
$$
\item  If $1\le p_0<\infty$, $w\in A_{p_0}$, and
$$T:L^{p_0}(\bbrn)\longrightarrow L^{p_0,\infty}(\bbrn),$$
then
$$T:L^{p_0}(w)\longrightarrow L^{p_0,\infty}(w).
$$
\end{enumerate}
\end{customthm}

The purpose of this paper is to continue this line of investigation in two directions. 
First, we show that the Stein--Soria--Weiss principle extends beyond weighted Lebesgue spaces, namely to homogeneous Herz spaces and to weighted Ces\`aro type spaces. 
Second, we establish multilinear analogues for operators satisfying the natural multilinear size condition, and we apply these results to several classes of rough singular integral operators.

\medskip

\subsection{Herz spaces} 

We begin with the Herz space extension. Let us recall the definition of weighted homogeneous Herz spaces.
\begin{definition}
Let $\alpha\in\bbr$, $0<p,q\le \infty$, and let $w$ be a weight on $\bbrn$. 
We define
$$\dot K_p^{\alpha,q}(w):=\Big\{ f\in L^p_{\mathrm{loc}}\big(\bbrn\setminus\{0\}\big):\Vert f\Vert_{\dot K_p^{\alpha,q}(w)} <\infty \Big\},$$
where
$$\Vert f\Vert_{\dot K_p^{\alpha,q}(w)}:=\bigg(\sum_{k\in\bbz} 2^{k\alpha q} \big\Vert \chi_{I_k}f\big\Vert_{L^p(w)}^q \bigg)^{1/q},$$
with the usual modification when $q=\infty$. In the unweighted case
$w\equiv 1$, we simply write $\dot K_p^{\alpha,q}(\bbrn)$.
\end{definition}

Herz spaces were introduced by C. Herz in \cite{h:h} and  many researchers have contributed to the theory (see, for example,  \cite{fw:fw, gh:gh, gly:gly}), since they have proved to be  quite useful in several problems in harmonic analysis, such as  to characterize the multipliers on Hardy spaces \cite{bs:bs}.

Our first result extends  Stein's theorem  to the setting of unweighted homogeneous Herz spaces.

\begin{theorem}\label{cor:Herz-power}
Let $T$ be given by \eqref{operatorT}, and assume that \eqref{sizecon} and
\eqref{LpboundT} hold for some $1<p_0<\infty$. Let $0<q\le \infty$. If
\begin{equation}\label{alphacond}
-\frac{n}{p_0}<\alpha<\frac{n}{p_0'},
\end{equation}
then
$$T:\dot K_{p_0}^{\alpha,q}(\bbrn)\longrightarrow \dot K_{p_0}^{\alpha,q}(\bbrn).$$
\end{theorem}
Here, $p_0'$ is the conjugate index of $p_0$.
In fact, this result will be an immediate consequence of the following one, which is the corresponding counterpart of Theorem \ref{SoWetheorem}.

\begin{theorem}\label{Herzextrapol}
Let $T$ be given by \eqref{operatorT}, and assume that \eqref{sizecon} and
\eqref{LpboundT} hold for some $1<p_0<\infty$. If $w\in A_{p_0}$ satisfies
\eqref{consondyad}, then for every $0<q\le \infty$,
$$T:\dot K_{p_0}^{0,q}(w)\longrightarrow \dot K_{p_0}^{0,q}(w).$$
\end{theorem}

We observe that to prove Theorem \ref{cor:Herz-power}, as a consequence of Theorem \ref{Herzextrapol}, it will be enough to take the power weight $w(x)=|x|^{p_0\alpha}$ and observe that   \eqref{alphacond} is equivalent to
$$-n<p_0\alpha<n(p_0-1),$$
and hence $w(x)=|x|^{p_0\alpha}\in A_{p_0}$. 
Moreover, for $x\in I_k$, we have $|x|\sim 2^k$, so
$$\big\Vert \chi_{I_k}f \big\Vert_ {L^{p_0}(w)}\sim 2^{k\alpha}\big\Vert \chi_{I_k}f \big\Vert_ {L^{p_0}(\bbrn)}.$$
Therefore
$$\Vert f\Vert_{\dot K_{p_0}^{0,q}(w)} \sim\Vert f\Vert_{\dot K_{p_0}^{\alpha,q}(\bbrn)}.$$
This proves Theorem \ref{cor:Herz-power} as a consequence of Theorem \ref{Herzextrapol}.

\hfill

\subsection{Ces\`aro spaces:} Concerning Ces\`aro spaces, let us consider the operator
\begin{equation}\label{uufxdef}
\UU f(x):=\frac{1}{|x|^n}\int_{|y|\le |x|} |f(y)| \,dy, 
\end{equation}
which is  a natural $n$-dimensional Hardy averaging operator. 
Up to the dimensional normalizing constant, it coincides with the radial averaging operator
$$f\mapsto \frac{1}{|B(0,|x|)|}\int_{B(0,|x|)} |f(y)|\,dy .$$
This operator was introduced by Faris \cite{f:f} and sharp $L^p(\bbrn)$ bounds for this operator were obtained by Christ and Grafakos \cite{Ch_Gr1995}.  
Let us  introduce the following spaces  which are closely related with the so-called Ces\`aro spaces (see, for example, \cite{am:am, am2:am2, lm:lm, s:s}). 

\begin{definition} Given a weight $w$, the weighted Ces\`aro space is defined by
$$
\mathcal C_p(w)=\big\{ f:  \UU f\in L^p(w)\big\},
$$
with the norm $\Vert f\Vert_{\mathcal C_p(w)}=\Vert \UU f\Vert_{L^p(w)}$.
\end{definition}

It is useful to observe that, whenever $w\in A_p$, one has $L^p(w)\subset \mathcal C_p(w)$. Thus, in situations where the full weighted estimate $T:L^p(w)\to L^p(w)$ cannot be expected from the size condition alone, the Ces\`aro target space provides a natural weaker replacement.

We shall also need the following local version
$$\mathcal{C}_{p, loc}(w)=\big\{ f:  \UU f \in L^p_{loc}(w)\big\}, $$ 
and, for given $0<r<p$, we define
$$\mathcal{C}_p^{(r)}(w)=\big\{ f:  \UU (|f|^r) \in L^{p/r}(w)\big\},$$
with the norm $\Vert f\Vert_{\mathcal C_p^{(r)}(w)}=\big\Vert \UU (|f|^r) \big\Vert_{L^{p/r}(w)}^{1/r}$.

Let us also define the operator
$$S_Tf(x):=\UU \big(Tf\big)(x).$$
In this context, we establish the following local weighted estimate for $T$.

\begin{theorem}\label{mainlinear1}
Let $T$ be given by \eqref{operatorT}, and assume that \eqref{sizecon} and \eqref{LpboundT} hold for some $1<p_0<\infty$. 
Then,  
$$T:L^{p_0}\log L(w)\longrightarrow \mathcal C_{p_0, loc}(w), \qquad \forall w\in A_1.$$
More precisely, for every $R>0$,
$$\bigg(\int_{B_R}|S_Tf(x)|^{p_0} w(x)\,dx \bigg)^{1/p_0} \lesssim \max\big\{1,w(B_R)^{1/p_0}\big\} \Vert f\Vert_{L^{p_0}\log L(w)}$$
where $B_R:=\{x\in\bbrn : |x|<R\}$ denotes the ball of radius $R$, centered at the origin.
\end{theorem}
 Here and below, the weighted Orlicz space $L^{p_0}\log L(w)$ is understood in the sense defined in Subsection \ref{lplogldef}.

\begin{theorem}\label{mainlinearpq}
Let $T$ be given by \eqref{operatorT}, and assume that \eqref{sizecon} and \eqref{LpboundT} hold for some $1<p_0<\infty$. 
Then, for every $p_0<p<\infty$,
$$T: L^p(w)\longrightarrow \mathcal C_p(w), \qquad \forall w\in A_{p/p_0}.$$
\end{theorem}

As an immediate consequence, we obtain the following result.
\begin{theorem}\label{mainlinearpq2}
Let $T$ be given by \eqref{operatorT}, and assume that \eqref{sizecon} holds.  If \eqref{LpboundT} holds for every $1<p_0<\infty$, then, for any $1<p<\infty$
$$T: L^p(w)\longrightarrow \mathcal C_p(w), \qquad \forall w\in A_p. $$
\end{theorem}

\medskip

\subsection{Multilinear extensions}
We now turn our attention to multilinear analogues. Let
\begin{equation}\label{bilinearTdef}
T\big(f_1,\dots,f_m\big)(x) := \int_{(\bbrn)^m} K(x,y_1,\dots,y_m)\bigg( \prod_{j=1}^{m}f_j(y_j) \bigg)\,dy_1\cdots dy_m
\end{equation}
be a multilinear operator whose kernel satisfies
\begin{equation}\label{bilinearsizecon}
\big|K(x,y_1,\dots,y_m)\big| \le \frac{B}{(|x-y_1|+\cdots+|x-y_m|)^{mn}}
\end{equation}
for some constant $B>0$.

Our first result is the natural multilinear extension of Stein's power-weight theorem under only the multilinear size condition and a single unweighted boundedness assumption.

\begin{theorem}\label{bilinearStein}
Let $T$ be given by \eqref{bilinearTdef}, and assume that
\eqref{bilinearsizecon} holds. Let $1<p_1,\dots,p_m<\infty$, and define $p$ by
$ 1/p=1/p_1+\cdots+1/p_m.$
Assume that
\begin{equation}\label{bilinearLpbound}
T:L^{p_1}(\bbrn)\times \cdots \times L^{p_m}(\bbrn) \longrightarrow L^{p}(\bbrn).
\end{equation}
Let 
\begin{equation}\label{con:powerweirang}
-n<\alpha_j<n(p_j-1), \qquad j=1,\dots,m.
\end{equation}
Then, 
$$ T:L^{p_1}(|x|^{\alpha_1})\times \cdots\times L^{p_m}(|x|^{\alpha_m}) \longrightarrow L^{p}(|x|^{\alpha_1 p/p_1+ \cdots+\alpha_m p/p_m } ).$$
\end{theorem}
Theorem \ref{bilinearStein} is recovered as the power-weight special case of the weighted Lebesgue estimate \eqref{eq:multiweightest} below, which itself follows from the weighted Herz-space estimate by taking $q_j=p_j$, $j=1,\ldots,m$. We now establish this weighted Herz-space estimate, which is the multilinear counterpart of Theorem \ref{Herzextrapol}.
\begin{theorem}\label{bilinearHerz}
Let $T$ be given by \eqref{bilinearTdef}, and assume that
\eqref{bilinearsizecon} holds. Let $1<p_1,\dots,p_m<\infty$, and define $p$ by
$ 1/p=1/p_1+\cdots+1/p_m.$
Assume that $T$ satisfies \eqref{bilinearLpbound}.
Let $w_j\in A_{p_j}$, $j=1,\dots,m$, satisfy
\begin{equation}\label{biconsondyad}
\esssup_{I_k} w_j \le C\,\essinf_{I_k^*} w_j, \qquad \quad k\in\bbz.
\end{equation}
Then, for every $0<q_1,\dots,q_m\le \infty$ and $0<q\le \infty$ with $1/q=1/q_1+\cdots+1/q_m,$
we have
\begin{equation}\label{tkp1q1kp2q2bdest}
 T:\dot K_{p_1}^{0,q_1}(w_1) \times \cdots\times \dot K_{p_m}^{0,q_m}(w_m) \longrightarrow \dot K_p^{0,q}(\nu)
 \end{equation}
where $\nu:=\prod_{j=1}^{m}w_j^{p/p_j}$.
\end{theorem}

Setting $q_j=p_j$, $j=1,\dots,m$, in Theorem \ref{bilinearHerz}, we have $q=p$. 
Since
$$\dot{K}_s^{0,s}(w)=L^s(w), \qquad 0<s<\infty,$$
with equality of quasi-norms, the conclusion \eqref{tkp1q1kp2q2bdest} becomes
\begin{equation}\label{eq:multiweightest}
T:L^{p_1}(w_1)\times\cdots\times L^{p_m}(w_m)
\longrightarrow L^p(\nu).
\end{equation}
This strong-type estimate is obtained under the assumption $1<p_j<\infty$, $j=1,\dots,m$, and hence does not cover endpoint cases in which some $p_j$ is equal to $1$.

Theorem \ref{bilinearStein} is now obtained by choosing
$$w_j(x)=|x|^{\alpha_j}, \qquad j=1,\dots,m.$$
The condition \eqref{con:powerweirang} is precisely equivalent to $w_j\in A_{p_j}$ for each $j=1,\dots,m$. 
Moreover, power weights satisfy the dyadic annular comparability condition \eqref{biconsondyad}.
Thus \eqref{eq:multiweightest}  implies Theorem \ref{bilinearStein},  since
$$\nu =\prod_{j=1}^m |x|^{\alpha_jp/p_j} = |x|^{\alpha_1p/p_1+\cdots+\alpha_mp/p_m}.$$

As another consequence of Theorem \ref{bilinearHerz}, we obtain the
following unweighted Herz-space estimate.
\begin{theorem}\label{cor:bilinearHerz}
Let $T$ be given by \eqref{bilinearTdef}, and assume that
\eqref{bilinearsizecon} holds. 
Let $1<p_1,\dots,p_m<\infty$, and define $p$ by $ 1/p=1/p_1+\cdots+1/p_m.$
Assume that $T$ satisfies \eqref{bilinearLpbound}. 
Let $0<q_1,\dots,q_m\le \infty$ and $0<q\le \infty$ with $1/q=1/q_1+\cdots+1/q_m$.
If
\begin{equation}\label{bialphacond}
-\frac{n}{p_j}<\alpha_j<\frac{n}{p_j'}, \qquad j=1,\dots,m,
\end{equation}
then
$$T: \dot K_{p_1}^{\alpha_1,q_1}(\bbrn) \times\cdots\times \dot K_{p_m}^{\alpha_m,q_m}(\bbrn) \longrightarrow \dot K_p^{\alpha_1+\cdots+\alpha_m,q}(\bbrn).$$
\end{theorem}
Indeed, \eqref{bialphacond} implies
$$ w_j(x):=|x|^{p_j\alpha_j}\in A_{p_j}, \qquad j=1,\dots,m. $$
Moreover, $$ \nu(x)=\prod_{j=1}^{m}w_j(x)^{p/p_j}=|x|^{p(\alpha_1+\cdots+\alpha_m)},$$
Thus, as in the proof of Theorem \ref{cor:Herz-power},
$$ \Vert f_j\Vert_ {\dot K_{p_j}^{0,q_j}(w_j)} \sim \Vert f_j\Vert_ {\dot K_{p_j}^{\alpha_j,q_j}(\bbrn)}, \qquad j=1,\dots,m,$$
and
$$ \big\Vert T(f_1,\dots,f_m)\big\Vert_ {\dot K_p^{0,q}(\nu)} \sim \big\Vert T(f_1,\dots,f_m)\big\Vert_ {\dot K_p^{\alpha_1+\cdots+\alpha_m,q}(\bbrn)}. $$
Now the conclusion follows from Theorem \ref{bilinearHerz}.

We also prove the following weak-type counterpart of \eqref{eq:multiweightest}, which includes endpoint cases where some of the exponents $p_j$ are equal to $1$.
\begin{theorem}\label{bilinearweakthm}
Let $T$ be given by \eqref{bilinearTdef}, and assume that \eqref{bilinearsizecon} holds. 
Let $1\le p_1,\dots,p_m<\infty$, and define $p$ by $1/p=1/p_1+\cdots+1/p_m.$
Assume that
\begin{equation}\label{biweakboundcon}
T:L^{p_1}(\bbrn)\times \cdots\times  L^{p_m}(\bbrn) \longrightarrow L^{p,\infty}(\bbrn).
\end{equation}
If $w_j\in A_{p_j}$, $j=1,\dots,m$, satisfy \eqref{biconsondyad}, 
then
$$ T:L^{p_1}(w_1)\times \cdots\times L^{p_m}(w_m) \longrightarrow L^{p,\infty}(\nu) $$
where $\nu=\prod_{j=1}^{m}w_j^{p/p_j}.$
\end{theorem}

For $0<r<\infty$, define
\begin{equation*}
S_T^r\big(f_1,\dots,f_m\big)(x):= \bigg(\frac1{|x|^n} \int_{|y|\le |x|}\big|T\big(f_1,\dots,f_m\big)(y)\big|^r\,dy \bigg)^{1/r}.
\end{equation*}
Then we have the following local multilinear logarithmic estimate. 
\begin{theorem}\label{thm:bilinear-local}
Let $T$ be given by \eqref{bilinearTdef}, and assume that \eqref{bilinearsizecon} holds.
Let $1< p_1,\dots,p_m<\infty$, and define $p$ by $1/p=1/p_1+\cdots+1/p_m.$
Assume that $T$ satisfies \eqref{bilinearLpbound}. 
Let $0<r<p$. 
If $w_j\in A_1$, $j=1,\dots,m$, then
$$T:L^{p_1}\log L(w_1)\times \cdots\times L^{p_m}\log L(w_m) \longrightarrow \mathcal{C}_{p, \mathrm{loc}}^{(r)}(\nu)$$
where $\nu=\prod_{j=1}^{m}w_j^{p/p_j}.$
More precisely, for every $R>0$,
\begin{align*}
&\bigg( \int_{B_R} \big|S_T^r\big(f_1,\dots,f_m\big)(x)\big|^p \nu(x) \, dx \bigg)^{1/p} \\
&\lesssim \prod_{j=1}^m \Big( \max\big\{1,w_j(B_R)\big\}^{1/p_j} \Vert f_j\Vert_ {L^{p_j}\log L(w_j)} \Big)
\end{align*}
where the implicit constant is independent of $R$.
\end{theorem}

We also obtain the following global multilinear weighted estimate.

\begin{theorem}\label{thm:bilinear-global}
Let $T$ be given by \eqref{bilinearTdef}, and assume that \eqref{bilinearsizecon} holds.
Let $1< p_1,\dots,p_m<\infty$, and define $p$ by $1/p=1/p_1+\cdots+1/p_m.$
Assume that $T$ satisfies \eqref{bilinearLpbound}. Let $0<r<p$.
Then, for every $q_j>p_j$, $j=1,\dots,m$, with $1/q=1/q_1+\cdots+1/q_m$, and for every $w_j\in A_{q_j/p_j}$, $j=1,\dots,m$, we have
$$ T: L^{q_1}(w_1) \times \cdots\times L^{q_m}(w_m) \longrightarrow \mathcal{C}_q^{(r)}(\nu)$$
where $\nu=\prod_{j=1}^{m}w_j^{q/q_j}.$

\end{theorem}

As a consequence, we have the following result. 
\begin{theorem}\label{thm:bilinear-global2}
Let $T$ be given by \eqref{bilinearTdef}, and assume that \eqref{bilinearsizecon} holds. Let $0<r\le 1/m$ be fixed.
If   \eqref{bilinearLpbound} holds for every $1<p_1,\dots,p_m<\infty$ and $1/m<p<\infty$ with $1/p=1/p_1+\cdots+1/p_m$, 
then  for every $1<q_1,\dots,q_m<\infty$  and $1/m<q<\infty$ with $1/q=1/q_1+\cdots+1/q_m$, and for every $w_j\in A_{q_j}$, $j=1,\dots,m$, 
we have
$$ T: L^{q_1}(w_1) \times\cdots\times L^{q_m}(w_m) \longrightarrow \mathcal{C}_q^{(r)}(\nu) $$
where $\nu=\prod_{j=1}^{m}w_j^{q/q_j}.$
\end{theorem}

In Section 2, we explain how these theorems apply to rough singular integrals, oscillatory rough singular integrals, and multilinear rough singular integrals. In several cases the usual weighted $A_p$ theory is already known; our point there is not to improve those results, but to show that they fit into a unified Stein-Soria-Weiss type mechanism and to obtain Herz and Ces\`aro space consequences.

\medskip

Let us briefly explain the main ideas behind the proofs. The basic strategy is to combine Stein's original localization argument with weighted estimates for
a radial Hardy-type operator. More precisely, after decomposing the input functions according to the dyadic annuli $I_k$, the action of the operator is split into a local part and an off-diagonal part. The local part is controlled by the assumed unweighted boundedness of $T$, together with the dyadic annular comparability condition on the weights. The off-diagonal part is estimated using only the size condition of the kernel and is dominated by a Hardy-type radial operator
$$\HH f(x)=\int_{\bbrn} \min\bigg\{\frac{1}{|x|^n},\frac{1}{|y|^n}\bigg\}|f(y)|\,dy.$$
Thus the heart of the argument is to obtain suitable weighted estimates for $\HH$, both on Lebesgue spaces and on Herz spaces. 
This approach also extends naturally to the multilinear setting, where the off-diagonal part is controlled by products of such Hardy-type operators.

For the Ces\`aro space estimates, a second idea is used. Since the Ces\`aro norm involves a radial averaging of $Tf$, we first convert this average, by Fubini's theorem, into weighted estimates for $T$ with respect to auxiliary weights built from the operator $\HH$. 
These auxiliary weights are radially nonincreasing $A_1$ weights and hence satisfy the dyadic annular comparability condition. This allows us to apply the Soria--Weiss theorem, or its multilinear analogue proved here, and then to pass to the desired weighted Ces\`aro estimates by Rubio de Francia extrapolation. In the local endpoint estimates, the logarithmic term arises from the well-known endpoint behavior of the Hardy-Littlewood maximal operator $M$: although $M$ is not bounded on $L^1$, its local $L^1$-averages can be controlled by an $L\log L$-type quantity via the standard rearrangement estimate.

\hfill

{\bf Organization.} In Section \ref{sec:2}, we present applications of the main results to rough singular integrals and related variants, first in the linear setting and then in the multilinear setting. Section \ref{sec:3} collects the necessary preliminary material on Muckenhoupt weights, extrapolation, weighted weak Lebesgue spaces, and $L^p\log L$ spaces. Section \ref{sec:4} contains the key estimates for the Hardy-type radial operator. The proofs of the linear results are given in Sections \ref{sec:5}--\ref{sec:8}, including the Herz space estimate and the local and global Ces\`aro space estimates. Finally, Sections \ref{sec:9}--\ref{sec:13} are devoted to the multilinear theory, including the multilinear Herz estimate, the weak-type estimate, and the local and global multilinear Ces\`aro space estimates.

\hfill

\section{Applications to rough singular integral operators}\label{sec:2}

We explain how known unweighted boundedness results, together with the size condition of the kernel, lead through our main theorems to weighted estimates and Herz space estimates for several classes of rough operators. The examples below include linear and multilinear rough singular integrals, oscillatory variants, and operators with radial amplitudes.

Let $\Omega\in L^1(\mathbb{S}^{mn-1})$ satisfy the cancellation condition
\begin{equation}\label{omecancon}
\int_{\mathbb{S}^{mn-1}}\Omega(\theta)\,d\sigma(\theta)=0,
\end{equation}
where $\sigma$ denotes the surface measure on the unit sphere $\mathbb{S}^{mn-1}$. 
We write
$$\mathfrak K_\Omega(y_1,\dots,y_m):= \frac{\Omega((y_1,\dots,y_m)')} {|(y_1,\dots,y_m)|^{mn}},$$
where
$(y_1,\dots,y_m)':=\frac{(y_1,\dots,y_m)}{|(y_1,\dots,y_m)|}\in\mathbb{S}^{mn-1}$.
Let $P$ be a real-valued measurable function on $(\bbrn)^{m+1}$, and let $h$ be a measurable function on $(\bbrn)^{m+1}$. 
Then we define the multilinear singular integral operator $\mathfrak{T}_{\Omega,h}^P$ by
\begin{equation}\label{general-rough-osc-amplitude}
\begin{aligned}
\mathfrak{T}_{\Omega,h}^P(f_1,\dots,f_m)(x) &:= \mathrm{p.v.}\int_{(\bbrn)^m} e^{iP(x,y_1,\dots,y_m)} h(x,y_1,\dots,y_m) \\
&\qquad \times \mathfrak K_\Omega(x-y_1,\dots,x-y_m) \prod_{j=1}^m f_j(y_j)\,dy_1\cdots dy_m .
\end{aligned}
\end{equation}
If $h \in L^{\infty}((\bbrn)^{m+1})$ and $\Omega \in L^{\infty}(\mathbb{S}^{mn-1})$, then the kernel of the operator $\mathfrak{T}_{\Omega,h}^P$ satisfies the size estimate
\begin{equation}\label{phkomkerest}
\begin{aligned}
&\big| e^{iP(x,y_1,\dots,y_m)} h(x,y_1,\dots,y_m) \mathfrak{K}_\Omega(x-y_1,\dots,x-y_m) \big|\\
&\qquad\lesssim \Vert h\Vert_{L^{\infty}((\bbrn)^{m+1})} \Vert \Omega\Vert_{L^{\infty}(\mathbb{S}^{mn-1})}\frac{1}{\big(|x-y_1|+\cdots+|x-y_m|\big)^{mn}}.
\end{aligned}
\end{equation}
Thus the size hypothesis in our main theorems is satisfied for this class of operators. The remaining input needed for our results is the corresponding unweighted boundedness, which depends on the particular form of $P$, $h$, and $\Omega$.

\subsection{The linear case $m=1$.} 
A frequently studied model is the radial amplitude case
$$T_{\Omega,h}^{P}f(x)=\mathrm{p.v.}\int_{\bbrn}e^{iP(x,y)}h(|x-y|)\frac{\Omega((x-y)')}{|x-y|^n}f(y)\,dy,$$
where $P$ is a real-valued polynomial on $\bbrn\times\bbrn$. 
This linear family contains several classical operators, depending on the choices of $P$ and $h$.

First, if $P\equiv 0$ and $h\equiv1$, then $T_{\Omega,h}^P$ reduces to the classical rough singular integral
$$ T_\Omega f(x)=\mathrm{p.v.}\int_{\bbrn}\frac{\Omega((x-y)')}{|x-y|^n}f(y)\,dy.$$
The study of $T_\Omega$ goes back to Calder\'on and Zygmund \cite{Ca_Zy1952, Ca_Zy1956}, who introduced the method of rotations and proved $L^p$ boundedness, $1<p<\infty$, for rough homogeneous kernels under appropriate assumptions on $\Omega$. 
In particular, the classical theory includes the case where $\Omega$ is odd and belongs to $L^1(\mathbb{S}^{n-1})$, and the case where $\Omega$ is even and belongs to $L\log L(\mathbb{S}^{n-1})$.

Second, if $h\equiv1$ but $P\not\equiv 0$, then $T_{\Omega,h}^P$ becomes the oscillatory rough singular integral
$$T_{\Omega}^Pf(x)= \mathrm{p.v.}\int_{\bbrn} e^{iP(x,y)} \frac{\Omega((x-y)')}{|x-y|^n}f(y)\,dy.$$
The $L^p$ theory of oscillatory singular integrals with polynomial phases began with Ricci and Stein \cite{RicciStein1987}, who considered smooth Calder\'on-Zygmund kernels. 
For rough kernels, $L^p$ boundedness was obtained by Lu and Zhang \cite{Lu_Zh1992} under conditions such as $\Omega\in L^q(\mathbb{S}^{n-1})$ for some $q>1$, and was later improved to $\Omega\in L\log L(\mathbb{S}^{n-1})$ by Jiang and Lu \cite{JiangLu1995}. 

Third, if $P\equiv 0$ but $h$ is nontrivial, one obtains
$$ T_{\Omega,h}f(x)=\mathrm{p.v.} \int_{\bbrn} h(|x-y|) \frac{\Omega((x-y)')}{|x-y|^n}f(y)\,dy.$$
This class was introduced by Fefferman \cite{Fe1979}, who proved $L^p$ boundedness under smoothness assumptions on $\Omega$ and boundedness of $h$. Namazi \cite{Na1986} later weakened the angular regularity assumption to $\Omega\in L^q(\mathbb{S}^{n-1})$, $q>1$. 
We also refer to \cite{Al_Pa2002, Duo_Ru1986} for further improvements and generalization.

Finally, if both $P$ and $h$ are nontrivial,  the available boundedness theory is more limited. In particular, one cannot at present treat arbitrary real-valued polynomial phases $P(x,y)$ in the same generality as in the case $h\equiv1$.
Recently, Al-Salman and Grafakos \cite{Al_Gr2026} obtained $L^p$ boundedness results for this type of operator under rough assumptions on $\Omega$ and $h$, but for rather special classes of phases. 
For example, they consider phases of the form
$$P(x,y)=|x|^\gamma(x\cdot y) \quad \text{or} \quad P(x,y)=|y|^\gamma(y\cdot x),\qquad 0\le \gamma\le 1,$$
under assumptions such as $\Omega\in L\log L(\mathbb{S}^{n-1})$ and $h\in L^\infty(0,\infty)$.

\medskip

Weighted estimates for rough oscillatory singular integrals have also been studied, but the available weight classes and the range of phases depend substantially on the form of the operator.

First, in the classical case $P\equiv 0$ and $h\equiv1$,  Duoandikoetxea and Rubio de Francia \cite{Duo_Ru1986} showed that if $\Omega\in L^\infty(\mathbb{S}^{n-1})$ satisfies the cancellation condition \eqref{omecancon} and $w\in A_p$, then
$$T_\Omega:L^p(w)\longrightarrow L^p(w).$$
More generally, weighted estimates for rough homogeneous singular integrals under weaker integrability assumptions on $\Omega$ were obtained by Watson \cite{Wa1990} and Duoandikoetxea \cite{Duoa1993}.

Second, when $h\equiv1$ but $P\not\equiv 0$,  weighted estimates in the usual Muckenhoupt scale are also known in several cases. 
In particular, sparse domination results for oscillatory rough singular integrals and their maximal truncations imply weighted $L^p(w)$ estimates for $A_p$ weights in the corresponding ranges; see, for example, Chen and Tao \cite{Ch_Ta2022} and Choudhary, Shrivastava, and Shuin \cite{Ch_Sh_Sh2025}. 
Thus, in the case $h\equiv1$, the weighted $A_p$ theory is relatively well developed, even in the presence of polynomial oscillatory factors.

Third, when $P=0$ but $h$ is nontrivial, Ojanen \cite{Oj2000} proved weighted $L^p$ estimates for rough kernels of the form
$K(y)=\frac{h(|y|)\Omega(y/|y|)}{|y|^n},$
where $h(|y|)$ is the radial factor and $\Omega(y/|y|)$ is the angular part of the kernel. 
However, the weights in Ojanen's theorem are not simply the usual Muckenhoupt $A_p$ weights; the condition is formulated in terms of rectangles arising from a stratified covering of a star-shaped set associated with $\Omega$. 
Thus, the weighted theory with a nontrivial radial factor is available, but it is not always expressed in the standard $A_p$ scale.

Finally, when both $P$ and $h$ are nontrivial, the available theory is significantly more limited. As mentioned above, even the unweighted boundedness of such operators is presently known only under rather restrictive assumptions on the phase and the radial factor. In particular, unlike the case $h\equiv 1$, there seems to be no general weighted $A_p$ theory for arbitrary real-valued polynomial phases $P(x,y)$ combined with a nontrivial rough radial factor $h$. Thus, while the results of Al-Salman and Grafakos \cite{Al_Gr2026} provide useful unweighted input for our theorem, no corresponding weighted $A_p$ estimates seem to be available at present in this level of generality.

Consequently, the role of our main theorems should be understood as follows. In cases where a full weighted $A_p$ theory is already known, such as the classical case $P\equiv 0$, $h\equiv1$, our conclusions are not meant to improve that theory. Rather, our theorem gives a uniform mechanism which, from unweighted $L^p$ boundedness and the size condition of the kernel, yields weighted estimates for dyadically comparable $A_p$ weights and, more generally, boundedness on the corresponding Herz scale. 
Moreover, even if the usual weighted estimate $T:L^p(w)\to L^p(w)$ is not available for $w\in A_p$, our Ces\`aro space theorem still gives a weaker weighted conclusion:
$$ T:L^p(w)\longrightarrow \mathcal{C}_p(w), \qquad w\in A_p,$$
provided the corresponding unweighted $L^s$ boundedness is known in the relevant range. Unlike the Herz space estimate, this Ces\`aro space conclusion does not require the dyadic annular comparability condition on the weight. This observation becomes particularly useful in settings, such as operators with a nontrivial radial factor, where the usual weighted theory is less complete.

\begin{theorem}
Let $1<p<\infty$, $0<q\le\infty$, and let $P(x,y)$ be a real-valued polynomial on $\bbrn\times\bbrn$. 
Let $\Omega\in L^\infty(\mathbb{S}^{n-1})$ satisfy the cancellation condition \eqref{omecancon}.
 If $w\in A_p$ satisfies the dyadic annular comparability condition \eqref{consondyad},
then
$$T_\Omega^P:\dot K_p^{0,q}(w)\longrightarrow \dot K_p^{0,q}(w).$$
 Moreover, if
$$-\frac{n}{p}<\alpha<\frac{n}{p'},$$
then
$$T_\Omega^P:\dot{K}_p^{\alpha,q}(\bbrn) \longrightarrow \dot{K}_p^{\alpha,q}(\bbrn).$$
\end{theorem}

\begin{theorem}
Let $1<p<\infty$, $0<q\le\infty$, and let $\Omega\in L^\infty(\mathbb{S}^{n-1})$ satisfy \eqref{omecancon}.
Let $h\in L^{\infty}(0,\infty)$, and let $0\le \gamma\le 1$. 
Suppose that the phase $P$ is one of the following two functions:
$$ P(x,y)=|x|^\gamma \langle x, y\rangle \quad \text{or} \quad P(x,y)=|y|^\gamma \langle y, x\rangle.$$
Then the following assertions hold.
\begin{enumerate}
\item If $w\in A_p$ satisfies the dyadic annular comparability condition \eqref{consondyad},
then
$$ T_{\Omega,h}^{P}:\dot{K}_p^{0,q}(w) \longrightarrow \dot{K}_p^{0,q}(w).$$
Moreover, if
$$-\frac{n}{p}<\alpha<\frac{n}{p'},$$
then
$$T_{\Omega,h}^P:\dot{K}_p^{\alpha,q}(\bbrn)\longrightarrow \dot{K}_p^{\alpha,q}(\bbrn).$$

\item For every $w\in A_p$, we have
$$T_{\Omega,h}^P:L^p(w)\longrightarrow \mathcal{C}_p(w).$$
\end{enumerate}
\end{theorem}
The assertions in part (1) follow from Theorem \ref{Herzextrapol}, while part (2) follows from Theorem \ref{mainlinearpq2}, since the corresponding unweighted $L^s$ boundedness, $1<s<\infty$, follows from the results cited above for these phases and amplitudes.

\medskip

\subsection{The multilinear case $m\ge 2$}
In contrast with the linear case, there are very few boundedness results for oscillatory rough singular integrals with a nontrivial amplitude in the multilinear setting. For this reason, we first focus on the basic multilinear rough singular integral, corresponding to $P\equiv 0$ and $h\equiv 1$.

In this case, the multilinear operator in \eqref{general-rough-osc-amplitude} becomes
\begin{equation*}
\mathfrak{T}_\Omega\big(f_1,\dots,f_m\big)(x) := \mathrm{p.v.} \int_{(\bbrn)^m} \mathfrak K_\Omega(x-y_1,\dots,x-y_m) \prod_{j=1}^m f_j(y_j)\,dy_1\cdots dy_m.
\end{equation*}
The boundedness theory for $\mathfrak{T}_\Omega$ was developed by Grafakos and Torres \cite{Gr_To2002} in the Calder\'on-Zygmund setting. 
In the rough kernel setting, Grafakos, He, and Honz\'ik \cite{Gr_He_Ho2018} proved $L^{p_1}\times L^{p_2}\to L^p$ boundedness in the bilinear case when $\Omega\in L^\infty(\mathbb{S}^{2n-1})$. Further improvements and extensions to the multilinear setting were obtained in \cite{Do_Sl2024,Do_Sl_Park2026,Gr_He_Ho_Park_JLMS,He_Park2023}.

Weighted estimates for multilinear rough singular integrals have also been studied by the second author \cite{Park_submitted} who obtained the following weighted estimate, which extends the classical linear result of Duoandikoetxea and Rubio de Francia \cite{Duo_Ru1986} to the multilinear setting.
\begin{customthm}{B}\label{singularthm}\cite{Park_submitted}
Let $1<p_1,\dots,p_m<\infty$, and let $1/m<p<\infty$ be defined by ${1}/{p}={1}/{p_1}+\cdots+{1}/{p_m}.$
Suppose that $w_j\in A_{p_j}$, $j=1,\dots,m$. Then we have
$$\big\Vert \mathfrak{T}_{\Omega}(f_1,\dots,f_m )\big\Vert_{L^p(\nu)} \lesssim\Vert \Omega\Vert_{L^{\infty}(\mathbb{S}^{mn-1})} \prod_{j=1}^{m}\Vert f_j\Vert_{L^{p_j}(w_j)},$$
where $\nu=\prod_{j=1}^{m}w_j^{p/p_j}.$
\end{customthm}

The original version in \cite{Park_submitted} is stated in the multiple-weight setting under general $L^q$ condition on $\Omega$, as a multilinear extension of  Watson \cite{Wa1990} and Duoandikoetxea \cite{Duoa1993}.
Although a weighted $L^p$ theory for $\mathfrak{T}_\Omega$ is already available, our main theorems also yield the following Herz space consequences
from its unweighted boundedness and kernel size estimate.

In view of \eqref{phkomkerest}, Theorem \ref{bilinearHerz} applies.
\begin{theorem}\label{cor:rough-multilinear-Herz}
Let $\Omega\in L^\infty(\mathbb{S}^{mn-1})$ satisfy \eqref{omecancon}.
Let $1<p_1,\dots,p_m<\infty$, and define $p$ by ${1}/{p}={1}/{p_1}+\cdots+{1}/{p_m}.$
Let $w_j\in A_{p_j}$, $j=1,\dots,m$, satisfy the dyadic annular comparability condition \eqref{biconsondyad}.
Then, for every $0<q_1,\dots,q_m\le\infty$ and $0<q<\infty$ with ${1}/{q}={1}/{q_1}+\cdots+{1}/{q_m},$
we have
$$ \mathfrak{T}_\Omega:\dot{K}_{p_1}^{0,q_1}(w_1) \times\cdots\times \dot{K}_{p_m}^{0,q_m}(w_m)\longrightarrow \dot{K}_p^{0,q}(\nu),$$
where $\nu=\prod_{j=1}^m w_j^{p/p_j}.$
\end{theorem}

Taking power weights in Theorem \ref{cor:rough-multilinear-Herz}, we obtain the following Herz space estimates.
\begin{corollary}\label{cor:rough-multilinear-unweighted-Herz}
Let $\Omega\in L^\infty(\mathbb{S}^{mn-1})$ satisfy \eqref{omecancon}.
Let $1<p_1,\dots,p_m<\infty$, and define $p$ by $1/p=1/{p_1}+\cdots+1/{p_m}.$
Let $0<q_1,\dots,q_m\le\infty$ and $0<q<\infty$ satisfy $1/q=1/{q_1}+\cdots+1/{q_m}.$
If
$$-\frac{n}{p_j}<\alpha_j<\frac{n}{p_j'},\qquad j=1,\dots,m,$$
then
$$ \mathfrak{T}_\Omega: \dot{K}_{p_1}^{\alpha_1,q_1}(\bbrn) \times\cdots\times \dot{K}_{p_m}^{\alpha_m,q_m}(\bbrn) \longrightarrow \dot{K}_p^{\alpha_1+\cdots+\alpha_m,q}(\bbrn).$$
\end{corollary}

\medskip

Finally, let us return to the general multilinear operator $\mathfrak T_{\Omega,h}^P$ in \eqref{general-rough-osc-amplitude}. 
If $h\in L^\infty((\bbrn)^{m+1})$ and $\Omega\in L^\infty(\mathbb{S}^{mn-1})$, then the kernel satisfies the multilinear size condition \eqref{phkomkerest}. Therefore, once the corresponding unweighted multilinear boundedness is available, our results can be applied in the same way.
More precisely, assume that
$$\mathfrak{T}_{\Omega,h}^{P}:L^{p_1}(\bbrn)\times\cdots\times L^{p_m}(\bbrn) \longrightarrow L^p(\bbrn), \qquad 1/p=1/{p_1}+\cdots+1/{p_m}.$$
Then Theorem \ref{bilinearHerz} gives
$$ \mathfrak{T}_{\Omega,h}^P: \dot{K}_{p_1}^{0,q_1}(w_1)\times\cdots\times \dot{K}_{p_m}^{0,q_m}(w_m) \longrightarrow \dot{K}_p^{0,q}(\nu)$$
for weights $w_j\in A_{p_j}$ satisfying \eqref{biconsondyad} and for any $0<q_1,\dots,q_m\le \infty$ with $1/q=1/q_1+\cdots+1/q_m$, 
where $\nu=\prod_{j=1}^m w_j^{p/p_j}$.

Similarly, if the corresponding unweighted boundedness is available in the full range required by Theorem \ref{thm:bilinear-global2}, then one obtains
the corresponding Ces\`aro space estimate even without the dyadic annular comparability condition \eqref{biconsondyad}.

\begin{remark}
Some papers use the term ``multilinear oscillatory singular integral'' for operators involving Taylor remainders or higher order commutator-type structures in  the form
$$ T_A f(x)= \mathrm{p.v.} \int_{\bbrn} e^{iP(x,y)} \frac{\Omega(x-y)}{|x-y|^{n+m}} R_{m+1}(A;x,y)f(y)\,dy, $$
where $R_{m+1}(A;x,y)$ is a Taylor remainder. 
This is different from the $m$-linear operator $\mathfrak T_{\Omega,h}^P$ considered above. 
Therefore, for ``genuinely multilinear'' oscillatory rough operators with amplitude, the conclusions above should be read as conditional applications of our main results, pending the corresponding unweighted $L^{p_1}\times\cdots\times L^{p_m}\to L^p$ boundedness theory.
\end{remark}

\hfill

\section{Preliminaries}\label{sec:3}

Let $M$  be the Hardy--Littlewood maximal operator, defined by
$$ Mf(x):=\sup_{B:  x\in B}\frac{1}{|B|}\int_B \big|f(y)\big|\,dy,$$
where the supremum is taken over all balls $B\subset \bbrn$ containing $x$.

\subsection{$A_p$ weights}\label{subsec:Apweights}
Let $w$ be a weight on $\bbrn$, that is, a nonnegative locally integrable function. 
For $1<p<\infty$, we say that $w\in A_p$ if
$$[w]_{A_p}:= \sup_B \bigg(\frac{1}{|B|}\int_B w(x)\,dx\bigg) \bigg(\frac1{|B|}\int_B w(x)^{-\frac{1}{p-1}} \, dx \bigg)^{p-1}<\infty,$$
where the supremum is taken over all balls $B\subset \bbrn$.
For $p=1$, we say that $w\in A_1$ if
$$[w]_{A_1}:=\sup_B \bigg(\frac1{|B|}\int_B w(x)\,dx\bigg) \Big(\essinf_{x\in B} w(x)\Big)^{-1}<\infty.$$
Equivalently,
\begin{equation}\label{a1character}
w\in A_1 \quad\Longleftrightarrow\quad Mw(x)\lesssim w(x) \quad \text{for a.e. } x\in\bbrn.
\end{equation}

We shall also use the standard notation
$$ A_\infty:=\bigcup_{1\le p<\infty}A_p.$$
It is known that if $w\in A_{\infty}$, then there exists $\delta>0$ such that 
\begin{equation}\label{ainftycha}
\frac{w(B')}{w(B)}\lesssim \bigg( \frac{|B'|}{|B|}\bigg)^{\delta} \quad \text{for all concentric balls $B'\subset B$}.
\end{equation}

One of the very useful properties of the Muckenhoupt classes is their connection with the Hardy--Littlewood maximal operator. More precisely, if $1<p<\infty$ and $w\in A_p$, then
\begin{equation}\label{weightHLmax}
    \Vert Mf\Vert_{L^p(w)} \le C \Vert f\Vert_{L^p(w)}.
\end{equation}
Moreover, if $w\in A_1$, then
\begin{equation*}
    \Vert Mf\Vert_{L^{1,\infty}(w)} \le C \Vert f\Vert_{L^1(w)}.
\end{equation*}
Here the constants depend only on the dimension, the exponent, and the corresponding Muckenhoupt constant of the weight.

The $A_p$ classes also have the following openness property.
\begin{lemma}\label{weightopen}
Let $1<p\le \infty$. If $w\in A_p$,  there exists $1<q<p$ such that $w\in A_q$.
\end{lemma}

We refer to \cite{gr:gr} for more details about properties of $A_p$ weights.

\medskip

\subsection{Extrapolation}

We recall the following form of Rubio de Francia extrapolation.
\begin{lemma}[\cite{Duo2011}]\label{rfrf}
Let $(f,g)$ be a pair of nonnegative measurable functions. Suppose that for some $1\le p_0<\infty$,
$$ \Vert g\Vert_ {L^{p_0}(v)} \le \varphi([v]_{A_1}) \Vert f\Vert_ {L^{p_0}(v)},  \qquad \text{for every } v\in A_1, $$
where $\varphi$ is a nondecreasing function on $[1,\infty)$.
Then, for every $p_0<p<\infty$ and every $v\in A_{p/p_0}$,
$$\Vert g\Vert_ {L^p(v)} \le \Phi([v]_{A_{p/p_0}}) \Vert f\Vert_ {L^p(v)},$$
where $\Phi$ is a nondecreasing function depending only on $p,p_0$, and $\varphi$.
\end{lemma}

We shall also use the following multilinear version, which follows from the limited range multilinear extrapolation theorem of Cruz-Uribe and Martell
\cite{Cr_Ma2018}. More precisely, it is obtained by taking $r_j^-=1$ and $r_j^+=\infty$ in Theorem 1.3 of \cite{Cr_Ma2018}, together with Remark 1.7 therein.

\begin{lemma}[\cite{Cr_Ma2018}]\label{multilinearRdf}
Let $m\ge 1$, and let  $(f,f_1,\ldots,f_m)$ be an $(m+1)$-tuple of nonnegative measurable functions.
Suppose that for some exponents $1\le p_1,\ldots,p_m<\infty$ with $1/p=1/p_1+\dots+1/p_m,$
$$ \Vert f\Vert_ {L^p(\nu)} \le C \prod_{j=1}^m \Vert f_j\Vert_ {L^{p_j}(w_j)} \quad \text{for every $w_1,\ldots,w_m\in A_1$}$$
where $\nu:=\prod_{j=1}^m w_j^{p/p_j}$.
Then, for every $q_1,\ldots,q_m$ with $p_j<q_j<\infty$,  $j=1,\ldots,m$, and $1/q=1/q_1+\dots+1/q_m,$ and for every collection of weights $v_1,\ldots,v_m$ with $v_j\in A_{q_j/p_j}$, $j=1,\ldots,m$, we have
$$\Vert f\Vert_ {L^q(\mu)} \le C \prod_{j=1}^m \Vert f_j\Vert_ {L^{q_j}(v_j)},$$
 where $\mu:=\prod_{j=1}^m v_j^{q/q_j}.$
The constant depends only on the dimension, the exponents, and the relevant Muckenhoupt constants of the weights.
\end{lemma}

\medskip

\subsection{Weighted weak Lebesgue spaces}\label{subsec:weakLp}
For a positive Borel measure $\mu$ on $\bbrn$ and a measurable function $f$ on $\bbrn$, we define the distribution function of $f$ with respect to $\mu$ by
$$d_f^\mu(\lambda):=\mu\bigl(\{x\in\mathbb R^n:|f(x)|>\lambda\}\bigr), \qquad \lambda>0.$$
For a weight $w$ on $\bbrn$, if $d\mu(x):=w(x)\,dx$, then we write
$$d_f^w(\lambda):=d_f^\mu(\lambda)=\int_{\{x\in\mathbb R^n:\,|f(x)|>\lambda\}} w(x)\,dx, \qquad \lambda>0,$$
and if $w\equiv 1$, we simply write $d_f^1(\lambda)=d_f(\lambda)$
For $0<p<\infty$, the weak Lebesgue space $L^{p,\infty}(\mu)$ consists of all measurable functions $f$ such that
$$\Vert f\Vert_ {L^{p,\infty}(\mu)} := \sup_{\lambda>0} \lambda\,\big(d_f^{\mu}(\lambda)\big)^{1/p} <\infty .$$
If $d\mu(x)=w(x)\,dx$, then we write
$$L^{p,\infty}(w):=L^{p,\infty}(\mu).$$

Concerning this space, we need to recall the following estimate. 

\begin{lemma}\label{lem:disjoint-lorentz} \cite{kk:kk}
Let $0<p<\infty$. Let $w$ be a weight on $\mathbb R^n$, and let $\{f_k\}_{k\in\bbz}$ be a family of functions whose supports are pairwise disjoint.
 Then for any $0<s\le p$,
$$ \Big\Vert\sum_{k\in\bbz} f_k\Big \Vert_{L^{p,\infty}(w)} \lesssim \bigg( \sum_{k\in\bbz}\big\Vert f_k\big\Vert_{L^{p,\infty}(w)}^s\bigg)^{1/s}$$
where the implicit constant depends only on $p,s$.
\end{lemma}

\medskip

\subsection{$L^p\log{L}(w)$ spaces}\label{lplogldef}
Let $\mu$ be a positive Borel measure on $\mathbb{R}^n$, and let $1\le p<\infty$. 
The space $L^p\log L(\mu)$ consists of all measurable functions $f$ such that
$$\int_{\mathbb{R}^n} \bigg( \frac{|f(x)|}{\lambda}\bigg)^{p}\log\bigg(e+\frac{|f(x)|}{\lambda}\bigg)\,d\mu(x)<\infty$$
for some $\lambda>0$.
We equip $L^p\log L(\mu)$ with the Luxemburg norm
$$ \Vert f\Vert_{L^p\log L(\mu)}= \inf\bigg\{ \lambda>0 : \int_{\bbrn}\bigg(\frac{|f(x)|}{\lambda}\bigg)^p \log \bigg(e+\frac{|f(x)|}{\lambda}\bigg) \,d\mu(x)\le 1 \bigg\}.$$
For a weight $w$, when $d\mu(x)=w(x) dx$, we write $L^p\log{L}(w)=L^p\log{L}(\mu)$ and it is well known, see for instance \cite{Be_Sh1988}, that
\begin{equation}\label{lloglcha}
\Vert f\Vert_ {L^p\log L(w)} \sim \bigg(\int_0^{\infty} \big(f^*_w(t) \big)^p\log\Big(e+\frac{1}{t} \Big)\, dt \bigg)^{1/p},
\end{equation}
where 
\begin{equation}\label{derearrangedef}
f_w^{*}(t):=\inf\big\{s\ge 0 : d_f^{w}(s)\le t \big\}, \quad t>0
\end{equation}
is the decreasing rearrangement of $f$ with respect to the weight $w$.

\hfill

\section{Key Estimates}\label{sec:4}

 We define a Hardy-type radial operator $\HH$
\begin{equation}\label{hardytype}
\HH f(x):=\int_{\bbrn}\min\bigg\{\frac{1}{|x|^n},\frac{1}{|y|^n}\bigg\} f(y)\,dy,
\qquad x\in \bbrn\setminus\{0\}.
\end{equation}
Clearly,
$$\HH f(x)=\UU f(x)+\VV f(x)$$
where $\UU$ is the Hardy averaging operator as in \eqref{uufxdef} and 
$$\VV f(x):=\int_{|y|>|x|}\frac{|f(y)|}{|y|^n}\,dy.$$
Since $\UU$ is pointwise controlled by the Hardy--Littlewood maximal operator, it follows immediately, using duality to estimate $\VV$, that for $1<p<\infty$ and $w\in A_p$,
\begin{equation}\label{hhlpwest}
\big\Vert \HH f\big\Vert_{L^p(w)}\lesssim \Vert f\Vert_{L^p(w)}.
\end{equation}

 The following proposition is an extension of \eqref{hhlpwest} to Herz spaces, which is a key estimate in the proof of Theorem \ref{Herzextrapol}. The result follows by Rubio de Francia's extrapolation theorem at least whenever $q\ge 1$. However, we include a direct proof. 
 
 \begin{proposition}\label{keylemHbdWi}
Let $1<p<\infty$ and $0<q\le \infty$. 
Assume that $w\in A_p$.
Then we have
$$\Vert \HH f\Vert_{\dot{K}_p^{0,q}(w)} \lesssim \Vert f\Vert_{\dot{K}_p^{0,q}(w)}.$$
\end{proposition}
\begin{proof}
We shall prove that there exists $\epsilon>0$ such that for each $k\in\bbz$
\begin{equation}\label{reducedclaim}
\big\Vert \chi_{I_k} \HH f\big\Vert_{L^p(w)}\lesssim \sum_{j\in\bbz}2^{-\epsilon |k-j|}\Vert \chi_{I_j}f\Vert_{L^p(w)} 
\end{equation}
where the constant in the inequality is independent of $k$.
Then we can conclude that
\begin{align*}
\Vert \HH f\Vert_{\dot{K}_p^{0,q}(w)}&=\Big\Vert \big\{ \Vert \chi_{I_k} \HH f \Vert_{L^p(w)} \big\}_{k\in\bbz}\Big\Vert_{\ell^q}\\
&\lesssim \Big\Vert \big\{ \Vert \chi_{I_j} f \Vert_{L^p(w)} \big\}_{j\in\bbz}\Big\Vert_{\ell^q} =\Vert f\Vert_{\dot{K}_p^{0,q}(w)}.
\end{align*}
where we applied H\"older's inequality if $q>1$ or the embedding $\ell^q\hookrightarrow \ell^1$ if $q\le 1$.\\

Now let us prove \eqref{reducedclaim}. 
For each $k\in\mathbb Z$, let $B_{2^k}:=B(0,2^k)$ denote the ball of radius $2^k$, centered at the origin, so that $I_k\subset B_{2^k}$.
For $x\in I_k$,
$$\HH f(x)\le \sum_{j\in\mathbb Z}\int_{I_j} \min\bigg\{ \frac1{|x|^n},\frac1{|y|^n}\bigg\}|f(y)|\,dy \lesssim \sum_{j\in\bbz} 2^{-n\max\{k,j\}}\Vert \chi_{I_j} f\Vert_{L^1}.$$
This yields that
\begin{equation}\label{eq:first-bk-estimate}
\big\Vert \chi_{I_k} \HH f\big\Vert_{L^p(w)} \lesssim \sum_{j\in\mathbb Z}2^{-n\max\{k,j\}}w(I_k)^{{1}/{p}} \Vert \chi_{I_j} f\Vert_{L^1}.
\end{equation}
Setting  $\sigma:=w^{-\frac1{p-1}}$ and applying Hölder's inequality,  for each $j\in\bbz$,
$$\Vert \chi_{I_j} f\Vert_{L^1}\le \big\Vert \chi_{I_j}f\big\Vert_{L^p(w)} \sigma(I_j)^{{1}/{p'}}.$$
Substituting this into \eqref{eq:first-bk-estimate} gives
\begin{equation*}
\big\Vert \chi_{I_k} \HH f\big\Vert_{L^p(w)} \lesssim \sum_{j\in\mathbb Z}2^{-n\max\{k,j\}}w(I_k)^{{1}/{p}}\sigma(I_j)^{{1}/{p'}}\big\Vert \chi_{I_j}f\big\Vert_{L^p(w)}.
\end{equation*}
Since $w\in A_p$, we have $\sigma\in A_{p'}$, and both $w$ and $\sigma$ belong to $A_\infty$, and hence, in view of \eqref{ainftycha}, we can take constants $\delta_w,\delta_\sigma>0$ and $C_w,C_\sigma>0$ such that for all concentric balls $B'\subset B$,
\begin{equation}\label{Ainfty}
\frac{w(B')}{w(B)} \le C_w\bigg(\frac{|B'|}{|B|}\bigg)^{\delta_w}, \qquad \frac{\sigma(B')}{\sigma(B)}\le C_\sigma\bigg(\frac{|B'|}{|B|}\bigg)^{\delta_\sigma}.
\end{equation}
In addition, the $A_p$ condition implies
\begin{equation}\label{Apball}
w(B_{2^l})^{{1}/{p}}\sigma(B_{2^l})^{{1}/{p'}}\le [w]_{A_p}^{{1}/{p}}|B_{2^l}| \sim 2^{nl}, \quad l\in\bbz.
\end{equation}

If $j\le k$, then using \eqref{Ainfty} and \eqref{Apball},
$$2^{-n\max\{k,j\}}w(I_k)^{{1}/{p}}\sigma(I_j)^{{1}/{p'}}\lesssim 2^{-nk} w(B_{2^k})^{{1}/{p}}\sigma(B_{2^k})^{{1}/{p'}}2^{-\delta_{\sigma}n(k-j)/{p'}}\sim 2^{-\delta_{\sigma}n(k-j)/{p'}}.$$
Similarly, if $j>k$, then
$$2^{-n\max\{k,j\}}w(I_k)^{{1}/{p}}\sigma(I_j)^{{1}/{p'}}\lesssim 2^{-\delta_{w}n(j-k)/{p}}.$$
Therefore, by taking
$$ \epsilon:=n \min\bigg\{\frac{\delta_\sigma}{p'},\frac{\delta_w}{p}\bigg\}>0,$$
the claim \eqref{reducedclaim} follows.
\end{proof}

\medskip

\begin{lemma}\label{utfl1uest}
There exists a constant $C>0$ such that for every $t>0$ and every weight $w \in A_1$,
\begin{equation}\label{extlemest2}
\big\Vert \chi_{B_t} Mf \big\Vert_{L^1(w)}\le C [w]_{A_1}\max\big\{1,w(B_t)\big\} \Vert f\Vert_{L\log L(w)}.
\end{equation}
\end{lemma}
\begin{proof}
By definition of the decreasing rearrangement in \eqref{derearrangedef},
$$\int_{B_t} \big|Mf(x)\big| w(x)\,dx \le \int_0^{w(B_t)} \big(Mf\big)^*_w(s)\,ds. $$
We first claim that
\begin{equation}\label{eq:Mt-reduction}
\int_0^{w(B_t)} \big(Mf\big)^*_w(s)\,ds
\le \max\big\{1,w(B_t)\big\}\int_0^1 \big(Mf\big)^*_w(s)\,ds.
\end{equation}
Indeed, if $w(B_t)\le 1$, then the estimate is immediate. 
If $w(B_t)>1$, since $(Mf)^*_w$ is nonincreasing, we have
$$ \frac1{w(B_t)}\int_0^{w(B_t)} \big(Mf\big)^*_w(s)\,ds \le \int_0^1 \big(Mf\big)^*_w(s)\,ds, $$
which yields \eqref{eq:Mt-reduction}.

Next, we use the standard rearrangement estimate for the Hardy--Littlewood maximal operator with respect to an $A_1$ weight:
$$ \big(Mf\big)^*_w(s)\lesssim [w]_{A_1}\,\frac{1}{s} \int_0^s  f^*_w(r) \,dr, \qquad s>0. $$
Therefore,
$$ \int_0^1 \big(Mf\big)^*_w(s)\,ds \lesssim [w]_{A_1}\int_0^1 \frac1s\int_0^s f^*_w(r)\,dr\,ds. $$
By Fubini's theorem and \eqref{lloglcha},
$$ \int_0^1 \frac1s\int_0^s f^*_w(r)\,dr\,ds=\int_0^1 f^*_w(r)\log{\frac{1}{r}}\, dr \lesssim \Vert f\Vert_ {L\log L(w)}.$$
Combining the above estimates, we obtain
$$ \int_{B_t} \big|Mf(x)\big|\,w(x)\,dx \lesssim [w]_{A_1}\max\big\{1,w(B_t)\big\}\,\Vert f\Vert_ {L\log L(w)}, $$
which proves \eqref{extlemest2}.
\end{proof}

\begin{lemma} \label{rdrd} 
If $w$ is a radially nonincreasing weight in $A_1$, then $w$ satisfies condition \eqref{consondyad}.
\end{lemma}
\begin{proof}
Fix $k\in\bbz$, and let $x\in I_k$ and $y\in I_k^*$. 
Then $|x|\sim |y|$ and $B(0,|x|)\subset B(0,4|y|)$.
Since $w$ is radially nonincreasing,   
$$ w(x) \lesssim \frac{w(B(0, |x|))}{|x|^n}, $$
and thus
\begin{align*}
w(x)&\lesssim \frac{1}{|x|^n}\int_{B(0,|x|)} w(z)\,dz \lesssim \frac{1}{|y|^n}\int_{B(0,4|y|)} w(z)\,dz\lesssim Mw(y).
\end{align*}
Since $w\in A_1$, we have
$$ Mw(y)\le [w]_{A_1}\,w(y) \quad \text{ for a.e. $y\in\bbrn$},$$
which yields
$$ w(x)\lesssim [w]_{A_1}\, w(y).
$$
Since this holds for a.e. $x\in I_k$ and for a.e. $y\in I_k^*$, we conclude that  $w$ satisfies \eqref{consondyad}.
\end{proof}

\begin{lemma} \label{ca1} 
If $w\in A_1$, then  $\HH w$ is a radially nonincreasing weight in  $A_1$ whenever $\HH w(x)<\infty$ at almost every point. 
Moreover
$$ [\HH w]_{A_1}\lesssim [w]_{A_1}.$$
\end{lemma}

\begin{proof}
Clearly $\HH w$ is radially nonincreasing and hence
\begin{align*}
M\big(\HH w\big)(x)&\lesssim \frac1{|x|^n}\int_{|y|\le |x|} \HH w(y)\,dy\\
&=\frac1{|x|^n}\int_{\bbrn} w(t) \int_{|y|\le |x|}\min\bigg\{\frac1{|y|^n},\frac1{|t|^n}\bigg\}\,dy dt.
\end{align*}
We split the inner integral according to whether $|t|\le |x|$ or $|t|>|x|$. Then the last expression is bounded by
\begin{align*}
\frac1{|x|^n}\int_{|t|\le |x|} w(t)\,dt  + \frac1{|x|^n}\int_{|y|\le |x|}\bigg(\frac1{|y|^n}\int_{|t|\le |y|} w(t)\,dt\bigg)dy + \int_{|t|>|x|}\frac{w(t)}{|t|^n}\,dt .
\end{align*}
The sum of the first and third terms is exactly  $\HH w(x)$.
Moreover, since
$$\frac1{|y|^n}\int_{|t|\le |y|} w(t)\,dt\lesssim Mw(y)\lesssim [w]_{A_1}w(y) \quad \text{ for a.e. $y\in\bbrn$},$$
we have
\begin{align*}
  \frac1{|x|^n}\int_{|y|\le |x|}\bigg(\frac1{|y|^n}\int_{|t|\le |y|} w(t)\,dt\bigg)dy &\lesssim [w]_{A_1}\frac1{|x|^n}\int_{|y|\le |x|} w(y)\,dy\le [w]_{A_1}\HH w(x).
\end{align*}
Combining the above estimates, we conclude that
$$M\big(\HH w\big)(x) \lesssim \HH w(x)+[w]_{A_1}\HH w(x)\lesssim [w]_{A_1}\HH w(x),$$
since $[w]_{A_1}\ge 1$. 
This completes the proof.
\end{proof}

\begin{lemma}\label{lem:weighthardyl1w}
Let $1\le p<\infty$ and $w\in A_p$. Then  there exists a constant $C>0$ such that for every nonnegative measurable function $f$ and for every
$R>0$,
$$ \HH f(x)\le C \frac{[w]_{A_p}^{1/p}}{w(B_R)^{1/p}}\Vert f\Vert_{L^p(w)} \qquad \text{whenever } |x|=R.$$
\end{lemma}

\begin{proof}
Fix $R>0$ and assume $|x|=R$. 
We write, by definition,
$$ \HH f(x) = \frac{1}{R^n}\int_{|y|\le R} f(y)\,dy + \int_{|y|>R}\frac{f(y)}{|y|^n}\,dy$$
and estimate the two terms separately.

{\bf Case 1.} $p=1$.

For the local part, since $w\in A_1$,
$$ \frac{w(B_R)}{|B_R|} \le [w]_{A_1}\essinf_{B_R}w.$$
Hence
\begin{align*}
\frac{1}{R^n}\int_{|y|\le R} f(y)\,dy & \lesssim \frac{[w]_{A_1}}{w(B_R)} \essinf_{B_R}w \int_{B_R} f(y)\,dy\\
&\le \frac{[w]_{A_1}}{w(B_R)} \int_{B_R} f(y)w(y)\,dy \le \frac{[w]_{A_1}}{w(B_R)}\Vert f\Vert_{L^1(w)}.
\end{align*}

For the outer part, if $|y|>R$, then $B_R\subset B(y,2|y|)$. Therefore
$$ \frac{w(B_R)}{|y|^n} \lesssim \frac{1}{|B(y,2|y|)|}\int_{B(y,2|y|)}w(z)\,dz \le Mw(y) \lesssim [w]_{A_1}w(y)$$
for a.e. $y$. Thus
$$\int_{|y|>R}\frac{f(y)}{|y|^n}\,dy \lesssim \frac{[w]_{A_1}}{w(B_R)}\int_{|y|>R} f(y)w(y)\,dy\le  \frac{[w]_{A_1}}{w(B_R)}\Vert f\Vert_{L^1(w)}.$$

\medskip

{\bf Case 2.} $1<p<\infty$.
Let $\sigma:=w^{-\frac{1}{p-1}}$ so that for any ball $B$ in $\bbrn$,
\begin{equation}\label{eq:apconswsi}
\frac{w(B)}{|B|}\bigg(\frac{\sigma(B)}{|B|} \bigg)^{p-1}\le [w]_{A_p} \quad \Longrightarrow \quad \sigma(B)^{\frac{p-1}{p}}\le [w]_{A_p}^{1/p}\frac{|B|}{w(B)^{1/p}}.
\end{equation}

For the local part, by H\"older's inequality and \eqref{eq:apconswsi},
$$\frac1{R^n}\int_{|y|\le R}f(y)\,dy \le \frac1{R^n}\|f\|_{L^p(w)}\sigma(B_R)^{\frac{p-1}{p}}  \lesssim \frac{[w]_{A_p}^{1/p}}{w(B_R)^{1/p}}\Vert f\Vert_{L^p(w)}.$$

On the other hand, we write the outer part as
$$\int_{|y|>R}\frac{f(y)}{|y|^n}\,dy=\sum_{k=1}^{\infty}\int_{A_k}\frac{f(y)}{|y|^n}\, dy\le \sum_{k=1}^{\infty}\frac{1}{(2^{k}R)^n}\int_{B_{2^kR}}f(y)\, dy $$
where $A_k:=\big\{ y\in\bbrn: 2^{k-1}R<|y|\le 2^kR\big\}$.
Using H\"older's inequality and \eqref{eq:apconswsi}, we have
$$\int_{B_{2^kR}}f(y)\, dy\le \Vert f\Vert_{L^p(w)}\sigma(B_{2^kR})^{\frac{p-1}{p}}\lesssim \frac{[w]_{A_p}^{1/p}}{w(B_{2^kR})^{1/p}}\Vert f\Vert_{L^p(w)}.$$
Since $w\in A_p\subset A_{\infty}$, there exists $\delta>0$ such that
\begin{equation*}
\frac{w(B_R)}{w(B_{2^kR})}\lesssim \bigg( \frac{|B_R|}{|B_{2^kR}|}\bigg)^{\delta}=2^{-kn\delta}\quad \Longrightarrow \quad \frac{1}{w(B_{2^kR})^{1/p}}\lesssim \frac{1}{w(B_R)^{1/p}}2^{-kn\delta /p}
\end{equation*}
Therefore, 
$$\int_{B_{2^kR}}f(y)\, dy\lesssim 2^{-kn\delta/p} \frac{[w]_{A_p}^{1/p}}{w(B_R)^{1/p}} \Vert f\Vert_{L^p(w)},$$
which finally implies
$$\int_{|y|>R}\frac{f(y)}{|y|^n}\,dy\lesssim  \frac{[w]_{A_p}^{1/p}}{w(B_R)^{1/p}} \Vert f\Vert_{L^p(w)}\sum_{k=1}^{\infty}2^{-kn\delta/p}\lesssim \frac{[w]_{A_p}^{1/p}}{w(B_R)^{1/p}} \Vert f\Vert_{L^p(w)},$$
as desired.
This completes the proof.
\end{proof}

\begin{proposition}\label{prop:multiweighthardyest}
Let $1\le p_1,\dots,p_m<\infty$ with $1/p_1+\cdots+1/p_m=1/p$.
Assume $w_j\in A_{p_j}$, $j=1,\dots,m$.
Then we have
$$ \bigg\Vert \prod_{j=1}^{m}\HH\big(|f_j|\big)\bigg\Vert_{L^{p,\infty}(\nu)} \lesssim \prod_{j=1}^{m}\Vert f_j\Vert_{L^{p_j}(w_j)}$$
where $\nu=\prod_{j=1}^{m}w_j^{p/p_j}$.
\end{proposition}

\begin{proof}
When $p_1,\dots,p_m>1$, the conclusion follows immediately from the embedding $L^p(\nu)\hookrightarrow L^{p,\infty}(\nu)$, H\"older's inequality, and the estimate \eqref{hhlpwest}.
For the case when at least one of  $p_j$ is equal to $1$,  we shall give a direct argument based on the radial monotonicity of the operator $\HH$ and the weighted pointwise Hardy-type estimate in Lemma \ref{lem:weighthardyl1w}.

Fix $\lambda>0$ and define
$$E_\lambda:=\bigg\{ x\in\bbrn : \prod_{j=1}^{m}\HH\big(|f_j|\big)(x)>\lambda\bigg\}.$$
We observe that $\prod_{j=1}^{m}\HH\big(|f_j|\big)$ is also radially nonincreasing, and thus $E_\lambda$ is a ball centered at the origin, possibly empty or all of $\bbrn$. 
Therefore, it is enough to consider balls $B_R\subset E_\lambda$. 
For such $R$, we have
\begin{equation}\label{eq:hf1hf2ptest}
 \prod_{j=1}^{m}\HH\big(|f_j|\big)(x)>\lambda \quad \text{ for all $x$ with $|x|=R$}.
\end{equation}
Moreover, by Hölder's inequality,
\begin{equation}\label{eq:nubrdecomp}
\nu(B_R)^{1/p} = \bigg(\int_{B_R}\prod_{j=1}^{m}w_j^{p/p_j}(y) \, dy \bigg)^{1/p} \le \prod_{j=1}^{m}w_j(B_R)^{1/p_j}.
\end{equation}
Using \eqref{eq:hf1hf2ptest} and \eqref{eq:nubrdecomp}, for any $x$ with $|x|=R$,
$$ \lambda \nu(B_R)^{1/p} \le \prod_{j=1}^{m}\Big( \HH\big(|f_j|\big)(x) w_j(B_R)^{1/p_j}\Big).$$
Now it follows from Lemma \ref{lem:weighthardyl1w} that the right-hand side is controlled by
$$\prod_{j=1}^{m}\Big( [w_j]_{A_{p_j}}^{1/p_j}\Vert f_j\Vert_{L^{p_j}(w_j)}\Big)$$
uniformly in $R$ and $x$ with $|x|=R$.
Taking the supremum over all $B_R\subset E_\lambda$ and using monotone convergence, we have
$$ \lambda \nu(E_\lambda)^{1/p} \lesssim \prod_{j=1}^{m}\Vert f_j\Vert_{L^{p_j}(w_j)} \quad \text{ uniformly in $\lambda>0$},$$
as desired.
\end{proof}

\hfill

\section{Proof of Theorem \ref{Herzextrapol}}\label{sec:5}

We first observe that the condition \eqref{LpboundT} implies
\begin{equation}\label{eq:local-unweighted-block}
\big\Vert \chi_{I_k}T\big(\chi_{I_k^*}f\big)\big\Vert_{L^{p_0}(\bbrn)}\le \big\Vert T\big(\chi_{I_k^*}f\big)\big\Vert_{L^{p_0}(\bbrn)} \lesssim \big\Vert\chi_{I_k^*}f\big\Vert_{L^{p_0}(\bbrn)}.
\end{equation}
We decompose
\begin{equation*}
Tf=T_{\mathrm{in}}f+T_{\mathrm{out}}f,
\end{equation*}
where
$$ T_{\mathrm{in}}f(x):=\sum_{k\in\bbz}\chi_{I_k}(x)\,T\big(\chi_{I_k^*}f\big)(x), \qquad T_{\mathrm{out}}f:=\sum_{k\in\bbz}\chi_{I_k}(x)\,T\big(\chi_{(I_k^*)^c}f\big)(x).$$
and will estimate the two terms separately.\\

Since
$$\chi_{I_k}T_{\mathrm{in}}f=\chi_{I_k}T\big(\chi_{I_k^*}f\big),$$
we obtain, by \eqref{consondyad} and \eqref{eq:local-unweighted-block},
\begin{align*}
&\big\Vert    \chi_{I_k} T_{\mathrm{in}}f \big\Vert_{L^{p_0}(w)}\le \Big(\sup_{x\in I_k}w(x)\Big)^{1/{p_0}} \big\Vert \chi_{I_k}T\big(\chi_{I_k^*}f\big)\big\Vert_{L^{p_0}(\bbrn)} \\
&\lesssim \Big(\inf_{x\in I_k^*}w(x)\Big)^{1/{p_0}}\big\Vert \chi_{I_k^*}f\big\Vert_{L^{p_0}(\bbrn)} \le \big\Vert \chi_{I_k^*}f\big\Vert_{L^{p_0}(w)} \le \sum_{|j-k|\le 1}\big\Vert \chi_{I_j}f\big\Vert_{L^{p_0}(w)}.
\end{align*}
Clearly, this yields
\begin{equation*}
\big\Vert T_{\mathrm{in}}f\big\Vert_{\dot{K}_{p_0}^{0,q}(w)} \lesssim \Vert f\Vert_{\dot{K}_{p_0}^{0,q}(w)}.
\end{equation*}

\medskip

On the other hand, we observe that for $x\in I_k$ and $y\in (I_k^*)^c$, 
$$|x-y|\sim |x|+|y|\sim \max\{|x|,|y|\}$$
and thus, by \eqref{sizecon}, for $x\in I_k$,
\begin{equation}\label{toutptest}
 \big|T\big(\chi_{(I_k^*)^c}f\big)(x)\big| \lesssim \HH (|f|)(x)
 \end{equation}
where the operator $\HH$ is defined in \eqref{hardytype}.
Therefore
\begin{equation*}
 \big|T_{\mathrm{out}}f(x)\big| \lesssim \HH (|f|)(x)
\end{equation*}
and it follows from Proposition \ref{keylemHbdWi} that for all $0<q\le \infty$
\begin{equation*}
\big\Vert T_{\mathrm{out}}f\big\Vert_{\dot{K}_{p_0}^{0,q}(w)}\lesssim \big\Vert \HH\big(|f|\big)\big\Vert_{\dot{K}_{p_0}^{0,q}(w)}\lesssim \Vert f\Vert_{\dot{K}_{p_0}^{0,q}(w)}.
\end{equation*}
This completes the proof. \qed

\hfill

\section{Proof of Theorem \ref{mainlinear1}}\label{sec:6}

For $x\in B_R$, let us define
$$ \UU_R\big( Tf\big)(x):=\frac{1}{|x|^n}\int_{|y|\le |x|, |y|<R} \big| Tf(y)\big|\, dy=:S_T^Rf(x). $$
Now we write
\begin{align*}
S_T^Rf(x)&\le \frac{1}{|x|^n}\sum_{k\le 1+\log_2{R}}\int_{|y|\le |x|, y\in I_k} \big| T\big(\chi_{I_k^*}f\big)(y)\big|\chi_{B_R}(y)\, dy\\
&\qq + \frac{1}{|x|^n}\sum_{k\le 1+\log_2{R}}\int_{|y|\le |x|, y\in I_k} \big| T\big(\chi_{(I_k^*)^c}f\big)(y)\big|\chi_{B_R}(y)\, dy\\
&=:S_T^{R, \mathrm{in}}f(x)+S_T^{R, \mathrm{out}}f(x).
\end{align*}

To estimate the term corresponding to $S_T^{R,\mathrm{in}}f$, we apply H\"older's inequality, which yields
$$S_T^{R,\mathrm{in}}f(x)\lesssim \bigg(\frac{1}{|x|^n}\sum_{k\le 1+\log_2{R}}\int_{|y|\le |x|, y\in I_k} \big| T\big(\chi_{I_k^*} f\big)(y)\big|^{p_0}\chi_{B_R}(y)\, dy \bigg)^{1/p_0}, $$
and consequently,
\begin{align*}
&\bigg(\int_{B_R} \big| S_T^{R, \mathrm{in}}f(x)\big|^{p_0} w(x)\, dx\bigg)^{1/p_0}\\
&\lesssim \bigg(\sum_{k\le 1+\log_2{R}}\int_{y\in I_k}\big| T\big(\chi_{I_k^*}f\big)(y)\big|^{p_0}\chi_{B_R}(y)\Gamma_Rw(y)\, dy \bigg)^{1/p_0}
\end{align*}
where 
\begin{equation}\label{gammatdef}
\Gamma_t w(y):=\int_{|y|\le |x|\le t}\frac{w(x)}{|x|^n}\, dx \quad \text{ for all $~t\ge |y|$.}
\end{equation}
Observe that $\Gamma_R w$ is a nonincreasing radial function. 
Thus, for any $y,z\in I_k^*$, we have
\begin{align*}
\Gamma_R w(y)&\le \int_{2^{k-2}\le |x|<2^{k+1}}\frac{w(x)}{|x|^n}\, dx + \int_{2^{k+1}\le |x|<R}\frac{w(x)}{|x|^n}\, dx\lesssim Mw(z)+\Gamma_R w(z).
\end{align*}
This implies 
\begin{equation}\label{supgammaest}
\esssup_{ I_k^*}\Gamma_R w \lesssim \essinf_{ I_k^*}\Big( Mw +\Gamma_{R} w\Big).
\end{equation}
Therefore, applying condition \eqref{LpboundT}, we deduce 
\begin{align*}
&\int_{y\in I_k}\big| T\big(\chi_{I_k^*}f\big)(y)\big|^{p_0}\chi_{B_R}(y) \Gamma_Rw(y)\, dy \\
&\le \Big( \esssup_{I_k}\Gamma_Rw\Big) \big\Vert T\big(\chi_{I_k^*}f\big) \big\Vert_{L^{p_0}(\bbrn)}^{p_0}\lesssim \Big( \esssup_{ I_k}\Gamma_Rw\Big) \big\Vert \chi_{I_k^*}f \big\Vert_{L^{p_0}(\bbrn)}^{p_0}\\
&\lesssim \int_{I_k^*}\big| f(y)\big|^{p_0} Mw(y)\, dy+\int_{I_k^*}\big| f(y)\big|^{p_0} \Gamma_Rw(y)\, dy.
\end{align*}
Summing over k, this yields
\begin{align*}
&\bigg(\int_{B_R} \big| S_T^{R, \mathrm{in}}f(x)\big|^{p_0} w(x)\, dx\bigg)^{1/{p_0}}\\
&\lesssim \bigg(\int_{B_{4R}}\big| f(y)\big|^{p_0} Mw(y)\, dy \bigg)^{1/{p_0}}+\bigg(\int_{B_{4R}}\big| f(y)\big|^{p_0} \Gamma_{4R} w(y)\, dy \bigg)^{1/{p_0}}.
\end{align*}
Since $w\in A_1$, using \eqref{a1character}, the first term can be controlled by $\Vert f\Vert_{L^{p_0}(w)}$, uniformly in $R$. Furthermore, by Fubini's theorem 
the second term is equal to
$$\Big\Vert \chi_{B_{4R}} \UU_{4R}\big(\chi_{B_{4R}}|f|^{p_0}\big)\Big\Vert_{L^1(w)}^{1/{p_0}}.$$
Since $\UU_{4R}\big( \chi_{B_{4R}}|f|^{p_0}\big)(x)\lesssim M\big(|f|^{p_0}\big)(x)$ uniformly in $R$, applying Lemma \ref{utfl1uest},
the last expression is bounded by
\begin{align*}
\max\big\{1,w(B_{4R})^{1/{p_0}}\big\} \big\Vert |f|^{p_0}\big\Vert_{L\log{L}(w)}^{1/{p_0}}\sim \max\big\{1,w(B_{R})^{1/{p_0}}\big\} \Vert f\Vert_{L^{p_0}\log{L}(w)}.
\end{align*}
Combining these estimates, we conclude that
$$\big\Vert \chi_{B_R}S_T^{R,\mathrm{in}}f\big\Vert_{L^{p_0}(w)}\lesssim \max\big\{1,w(B_{R})^{{1}/{{p_0}}}\big\} \Vert f\Vert_{L^{p_0}\log{L}(w)}.$$

On the other hand, in view of \eqref{toutptest}, for $y\in I_k$
\begin{equation*}
\big| T\big(\chi_{(I_k^*)^c}f\big)(y)\big| \lesssim \HH \big(|f|\big)(y)
\end{equation*}
and thus
\begin{align*}
S_T^{R,\mathrm{out}}f(x)&\lesssim \frac{1}{|x|^n}\int_{|y|\le |x|}   \chi_{B_R}(y)\HH \big( |f|\big)(y)   \, dy\lesssim M\Big(\chi_{B_R}\HH \big( |f|\big) \Big)(x).
\end{align*}
Since $w\in A_1\subset A_{p_0}$, we can invoke \eqref{weightHLmax} and \eqref{hhlpwest} to deduce
\begin{align*}
\bigg(\int_{B_R} \big| S_T^{R,\mathrm{out}}f(x)\big|^{p_0} w(x)\, dx\bigg)^{1/p_0}&\lesssim \Big\Vert M\Big(\chi_{B_R} \HH \big( |f|\big) \Big)\Big\Vert_{L^{p_0}(w)}\\
&\lesssim \big\Vert \HH\big(|f|\big)\big\Vert_{L^{p_0}(w)} \lesssim \Vert f\Vert_{L^{p_0}(w)},
\end{align*}
uniformly in $R>0$.
This concludes the proof. \qed

\hfill

\section{Proof of Theorem \ref{mainlinearpq}}\label{sec:7}

Let $u\in A_1$, and fix a number $\alpha$ such that $0<\alpha<n$. 
For each $N\in \mathbb{N}$, define
$$ u_N(x):=\min\big\{u(x),N|x|^{-\alpha}\big\}, \qquad x\in \bbrn.$$
Then $u_N(x)\nearrow u(x)$ for almost every $x\in \mathbb{R}^n$ as $N\to\infty$.

We first claim that $u_N\in A_1$, uniformly in $N$. Indeed, it is well known that $|x|^{-\alpha}\in A_1$ whenever $0<\alpha<n$. Hence, for every ball $B$,
\begin{align*}
\frac{1}{|B|}\int_B u_N(x)\,dx &\le \min\bigg\{\frac{1}{|B|}\int_B u(x)\,dx, N\frac{1}{|B|}\int_B |x|^{-\alpha}\,dx\bigg\} \\
&\le \max\Big\{ [u]_{A_1},\big[|x|^{-\alpha}\big]_{A_1}\Big\} \min\Big\{\essinf_{B}u,\, N\essinf_{B}|x|^{-\alpha}\Big\} \\
&=\max\Big\{[u]_{A_1},\big[|x|^{-\alpha}\big]_{A_1}\Big\}\essinf_{B}u_N.
\end{align*}
Thus $u_N\in A_1$, and
\begin{equation}\label{una1ua1uniform}
 [u_N]_{A_1}\le \max\Big\{[u]_{A_1},\big[|x|^{-\alpha}\big]_{A_1}\Big\}\lesssim [u]_{A_1}
 \end{equation}
where the last inequality follows from $[u]_{A_1}\ge 1$.

Next we show that $\HH u_N(x)<\infty$ for almost every $x\in \bbrn$.
Since $u_N(x)\le N|x|^{-\alpha}$, we have
$$\HH u_N(x) \le N \HH\big(|\cdot|^{-\alpha}\big)(x).$$
For $x\neq 0$, by the definition of $\HH$,
\begin{align*}
\HH\big(|\cdot|^{-\alpha}\big)(x) &= \frac1{|x|^n}\int_{|y|\le |x|}|y|^{-\alpha}\,dy +\int_{|y|>|x|} \frac{1}{|y|^{n+\alpha}}\, dy \lesssim |x|^{-\alpha}.
\end{align*}
Hence
\begin{equation}\label{hhunxfinite}
\HH u_N(x)\lesssim N|x|^{-\alpha}<\infty \qquad \text{for a.e. }x\in\bbrn.
\end{equation}
Therefore, by Lemma \ref{ca1},
\begin{equation*}
\HH u_N\in A_1 \qquad \text{ and } \qquad [ \HH u_N]_{A_1}\lesssim [u_N]_{A_1}\lesssim [u]_{A_1}~~\text{uniformly in $N$}.
\end{equation*}

\medskip

Now, using H\"older's inequality and Fubini's theorem, we obtain
\begin{align*}
\int_{\bbrn}\big|S_Tf(x)\big|^{p_0} u_N(x)\,dx&\le \int_{\bbrn}\bigg(\frac1{|x|^n}\int_{|y|\le |x|}\big|Tf(y)\big|^{p_0}\,dy\bigg) u_N(x)\,dx \\
&=\int_{\bbrn}\big|Tf(y)\big|^{p_0} \bigg(\int_{|x|\ge |y|}\frac{u_N(x)}{|x|^n}\,dx\bigg)\,dy \\
&\le \int_{\bbrn} \big|Tf(y)\big|^{p_0} \HH u_N(y)\,dy.
\end{align*}
Since $u_N\in A_1$ and $\HH u_N<\infty$ almost everywhere, Lemma \ref{ca1} implies that $\HH u_N$ is a radially nonincreasing weight in $A_1$. Hence, by Lemma \ref{rdrd}, the weight $\HH u_N$ satisfies condition \eqref{consondyad}. 
Then we may apply Theorem \ref{SoWetheorem}, together with the assumption \eqref{LpboundT}, to obtain
$$ \int_{\bbrn} \big|Tf(x)\big|^{p_0} \HH u_N(x)\,dx \lesssim \int_{\bbrn} \big|f(x)\big|^{p_0} \HH u_N(x)\,dx.$$
By the symmetry of the kernel of $\HH$, another application of Fubini's theorem yields
$$ \int_{\bbrn} |f(x)|^{p_0} \HH u_N(x)\,dx= \int_{\bbrn} u_N(x)\HH\big(|f|^{p_0}\big)(x)\,dx \le \int_{\bbrn} u(x)\HH\big(|f|^{p_0}\big)(x)\,dx.$$
Combining the above estimates, we conclude that
$$\int_{\bbrn}\big|S_Tf(x)\big|^{p_0} u_N(x)\,dx \lesssim \int_{\bbrn} u(x) \HH\big(|f|^{p_0}\big)(x)\,dx,$$
with an implicit constant independent of $N$. 
Letting $N\to\infty$ and applying the monotone convergence theorem on the left-hand side, we obtain
\begin{equation}\label{weightedesta1}
\int_{\bbrn}\big|S_Tf(x)\big|^{p_0} u(x)\,dx \lesssim \int_{\bbrn} u(x) \HH\big(|f|^{p_0}\big)(x)\,dx \qquad \text{for every }u\in A_1.
\end{equation}

\hfill

We now apply Lemma \ref{rfrf} to \eqref{weightedesta1} with the pair 
$$\Big( \HH \big(|f|^{p_0}\big)^{1/p_0},\, S_Tf \Big)$$
to obtain that, for every $p>p_0$ and every weight $w\in A_{p/p_0}$,
$$\Vert Tf\Vert_{\mathcal{C}_p(w)}= \big\Vert S_T f\big\Vert_{L^p(w)} \lesssim \Big\Vert \big(\HH(|f|^{p_0})\big)^{1/p_0}\Big\Vert_{L^p(w)}=\big\Vert \HH(|f|^{p_0})\big\Vert_{L^{p/p_0}(w)}^{1/p_0},$$
Since $p/p_0>1$ and $w\in A_{p/p_0}$, \eqref{hhlpwest} yields
$$ \big\Vert \HH\big(|f|^{p_0}\big)\big\Vert_{L^{p/p_0}(w)} \lesssim \big\Vert |f|^{p_0} \big\Vert_{L^{p/p_0}(w)}=\Vert f\Vert_{L^p(w)}^{p_0}.$$
Therefore, 
$$ \big\Vert Tf\big\Vert_{\mathcal{C}_p(w)} \lesssim \Vert f\Vert_{L^p(w)},$$
and the result follows. \qed

\hfill

\section{Proof of Theorem \ref{mainlinearpq2}}\label{sec:8}

By hypothesis \eqref{LpboundT} holds for every $p_0>1$. Let $p>1$ and    $w\in A_p$. 
By Lemma \ref{weightopen}, there exists   $1<t<p$ such that $w\in A_t$. Set $p_0:={p}/{t}.$
Then $1<p_0<p$ and $w\in A_{p/{p_0}}$. Therefore, the desired estimate follows immediately from Theorem \ref{mainlinearpq}. \qed

\hfill

\section{Proof of Theorem \ref{bilinearHerz}}\label{sec:9}

We first note that, by the assumption \eqref{biconsondyad}
\begin{equation}\label{nu-local-control}
\esssup_{I_k}\nu \lesssim \prod_{j=1}^{m}\Big(\essinf_{ I_k^*}w_j\Big)^{{p}/{p_j}}.
\end{equation}

We define
\begin{equation*}
T_{\mathrm{\mathrm{in}}}\big(f_1,\dots,f_m\big)(x):=\sum_{k\in\bbz}\chi_{I_k}(x)T\big(\chi_{I_k^*}f_1,\dots,\chi_{I_k^*}f_m\big)(x),
\end{equation*}
and
\begin{equation*}
T_{\mathrm{out}}\big(f_1,\dots,f_m\big)(x):=T\big(f_1,\dots,f_m\big)(x)-T_{\mathrm{in}}\big(f_1,\dots,f_m\big)(x)
\end{equation*}
so that
\begin{equation}\label{Tdecompinout}
T(f_1,\dots,f_m)=T_{\mathrm{in}}(f_1,\dots,f_m)+T_{\mathrm{out}}(f_1,\dots,f_m).
\end{equation}

\medskip

To deal with the first term, we observe that the assumed unweighted boundedness \eqref{bilinearLpbound} implies that for each $k\in\bbz$, 
\begin{equation}\label{bilinearestonik}
\big\Vert \chi_{I_k}T\big(\chi_{I_k^*}f_1,\dots,\chi_{I_k^*}f_m)\big\Vert_{L^p(\bbrn)} \lesssim \prod_{j=1}^{m}\big\Vert \chi_{I_k^*}f_j\big\Vert_{L^{p_j}(\bbrn)}.
\end{equation}
Now we write
$$ \big\Vert T_{\mathrm{in}}(f_1,\dots,f_m)\big\Vert_{\dot{K}_p^{0,q}(\nu)}= \Big\Vert \Big\{ \big\Vert \chi_{I_k}T\big(\chi_{I_k^*}f_1,\dots,\chi_{I_k^*}f_m\big)\big\Vert_{L^p(\nu)} \Big\}_{k\in\bbz}\Big\Vert_{\ell^q}$$
and apply \eqref{nu-local-control} and \eqref{bilinearestonik} to obtain
\begin{align*}
\big\Vert \chi_{I_k}T\big(\chi_{I_k^*}f_1,\dots,\chi_{I_k^*}f_m\big)\big\Vert_{L^p(\nu)}&\le \Big(\esssup_{ I_k}\nu\Big)^{1/p} \big\Vert \chi_{I_k}T\big(\chi_{I_k^*}f_1,\dots,\chi_{I_k^*}f_m)\big\Vert_{L^p(\bbrn)} \\
&\lesssim \prod_{j=1}^{m}\bigg(\Big(\essinf_{ I_k^*}w_j\Big)^{1/p_j}\big\Vert \chi_{I_k^*}f_j\big\Vert_{L^{p_j}(\bbrn)}  \bigg)\\
&\lesssim \prod_{j=1}^{m}\big\Vert \chi_{I_k^*}f_j\big\Vert_{L^{p_j}(w_j)}.
\end{align*}
Therefore we have
\begin{align*}
\big\Vert T_{\mathrm{in}}(f_1,\dots,f_m)\big\Vert_{\dot{K}_p^{0,q}(\nu)}&\lesssim \Big\Vert \Big\{ \prod_{j=1}^{m} \big\Vert \chi_{I_k^*}f_j\big\Vert_{L^{p_j}(w_j)}\Big\}_{k\in\bbz} \Big\Vert_{\ell^{q}}\\
&\le  \prod_{j=1}^{m}\Big\Vert \Big\{  \big\Vert \chi_{I_k^*}f_j\big\Vert_{L^{p_j}(w_j)}\Big\}_{k\in\bbz} \Big\Vert_{\ell^{q_j}}
\end{align*}
where the last inequality follows from H\"older's inequality, with the usual modification when some $q_j=\infty$.
Since $I_k^*=I_{k-1}\cup I_k\cup I_{k+1}$, for each $j=1,\dots,m$,
$$\Big\Vert \Big\{  \big\Vert \chi_{I_k^*}f_j\big\Vert_{L^{p_j}(w_j)}\Big\}_{k\in\bbz} \Big\Vert_{\ell^{q_j}}\sim \Big\Vert \Big\{  \big\Vert \chi_{I_k}f_j\big\Vert_{L^{p_j}(w_j)}\Big\}_{k\in\bbz} \Big\Vert_{\ell^{q_j}}=\Vert f_j\Vert_{\dot{K}_{p_j}^{0,q_j}(w_j)},$$
and thus
\begin{equation}\label{inHerzfinal}
\big\Vert T_{\mathrm{in}}(f_1,\dots,f_m)\big\Vert_{\dot{K}_p^{0,q}(\nu)}\lesssim \prod_{j=1}^{m}\Vert f_j\Vert_{\dot{K}_{p_j}^{0,q_j}(w_j)},
\end{equation}
as desired.

\medskip

Moreover, we observe that 
if $x\in I_k$ and at least one of $y_j$, $j=1,\dots,m$, does not belong to $I_k^*$, then
$$\sum_{j=1}^{m} |x-y_j| \sim\max\{|x|,|y_1|,\dots,|y_m|\}.$$
Hence
\begin{align*}
&\big|T_{\mathrm{out}}\big(f_1,\dots,f_m\big)(x)\big|\\
&\lesssim \int_{(\bbrn)^m} \frac{1}{\max\{|x|,|y_1|,|y_2|,\dots,|y_m|\}^{mn}}\prod_{j=1}^{m}|f_j(y_j)| \,dy_1\dots dy_m.
\end{align*}
Since
$$ \frac1{\max\{|x|,|y_1|,|y_2|,\dots,|y_m|\}^{mn}} \le \prod_{j=1}^{m}\min\bigg\{ \frac{1}{|x|^n},\frac{1}{|y_j|^n}\bigg\},$$
it follows that
\begin{equation}\label{outpointwise}
\big|T_{\mathrm{out}}\big(f_1,\dots,f_m\big)(x)\big| \lesssim \prod_{j=1}^{m}\HH\big(|f_j|\big)(x),
\end{equation}
where the operator $\HH$ is the Hardy-type radial operator, defined in \eqref{hardytype}.
Then H\"older's inequality yields that
\begin{equation*}
\big\Vert T_{\mathrm{out}}(f_1,\dots,f_m)\big\Vert_{\dot{K}_p^{0,q}(\nu)}\lesssim \bigg\Vert \prod_{j=1}^{m}\HH\big(|f_j|\big)\bigg\Vert_{\dot{K}_p^{0,q}(\nu)}\le \prod_{j=1}^{m}\Big\Vert \HH\big(|f_j|\big)\Big\Vert_{\dot{K}_{p_j}^{0,q_j}(w_j)}.
\end{equation*}
Since $w_j\in A_{p_j}$ and $p_j>1$, Proposition \ref{keylemHbdWi} yields
$$ \big\Vert \HH(|f_j|)\big\Vert_{\dot{K}_{p_j}^{0,q_j}(w_j)} \lesssim \Vert f_j\Vert_{\dot{K}_{p_j}^{0,q_j}(w_j)}, \qquad j=1,\dots,m.$$
Therefore
\begin{equation}\label{outHerzfinal}
\big\Vert T_{\mathrm{out}}(f_1,\dots,f_m)\big\Vert_{\dot{K}_{p}^{0,q}(\nu)} \lesssim \prod_{j=1}^{m}\Vert f_j\Vert_{\dot{K}_{p_j}^{0,q_j}(w_j)}.
\end{equation}

Combining \eqref{inHerzfinal} and \eqref{outHerzfinal}, we conclude that
$$\big\Vert T(f_1,\dots,f_m)\big\Vert_{\dot{K}_{p}^{0,q}(\nu)} \lesssim \prod_{j=1}^{m}\Vert f_j\Vert_{\dot{K}_{p_j}^{0,q_j}(w_j)}.$$
This completes the proof. \qed

\hfill

\section{Proof of Theorem \ref{bilinearweakthm}}\label{sec:10}

In view of the decomposition \eqref{Tdecompinout}, it suffices to prove
\begin{equation}\label{Lorentzinred}
\big\Vert T_{\mathrm{in}}(f_1,\dots,f_m)\big\Vert_{L^{p,\infty}(\nu)}\lesssim \prod_{j=1}^{m}\Vert f_j\Vert_{L^{p_j}(w_j)}
\end{equation}
and
\begin{equation*}
\big\Vert T_{\mathrm{out}}(f_1,\dots,f_m)\big\Vert_{L^{p,\infty}(\nu)}\lesssim \prod_{j=1}^{m}\Vert f_j\Vert_{L^{p_j}(w_j)}.
\end{equation*}

\medskip

First of all, Lemma \ref{lem:disjoint-lorentz} yields
\begin{equation}\label{eq:Tin-first}
\big\Vert T_{\mathrm{in}}(f_1,\dots,f_m)\big\Vert_{L^{p,\infty}(\nu)} \lesssim \bigg( \sum_{k\in\bbz} \big\Vert \chi_{I_k}T\big(\chi_{I_k^*}f_1,\dots, \chi_{I_k^*}f_m\big)\big\Vert_{L^{p,\infty}(\nu)}^p\bigg)^{1/p}.
\end{equation}
Moreover, using \eqref{nu-local-control} and \eqref{biweakboundcon}, for each $k\in\mathbb Z$ we obtain
\begin{align*}
&\big\Vert \chi_{I_k}T\big(\chi_{I_k^*}f_1,\dots,\chi_{I_k^*}f_m)\big\Vert_{L^{p,\infty}(\nu)}\\
&\le \Big(\esssup_{I_k}\nu\Big)^{1/p} \big\Vert \chi_{I_k}T\big(\chi_{I_k^*}f_1,\dots,\chi_{I_k^*}f_m\big)\big\Vert_{L^{p,\infty}(\bbrn)}\\
&\lesssim \prod_{j=1}^{m}\bigg( \Big(\essinf_{I_k^*}w_j\Big)^{{1}/{p_j}} \big\Vert \chi_{I_k^*}f_j\big\Vert_{L^{p_j}(\bbrn)} \bigg)\\
&\le \prod_{j=1}^{m}\big\Vert  \chi_{I_k^*}f_j\big\Vert_{L^{p_j}(w_j)}.
\end{align*}
Substituting this into \eqref{eq:Tin-first}, and applying H\"older's inequality, we obtain
\begin{align*}
\big\Vert T_{\mathrm{in}}(f_1,\dots,f_m)\big\Vert_{L^{p,\infty}(\nu)} &\lesssim \bigg( \sum_{k\in\bbz} \Big(\prod_{j=1}^{m}\big\Vert  \chi_{I_k^*}f_j\big\Vert_{L^{p_j}(w_j)}\Big)^p  \bigg)^{1/p}\\
&\le \prod_{j=1}^{m}\bigg(\sum_{k\in\bbz}\big\Vert  \chi_{I_k^*}f_j\big\Vert_{L^{p_j}(w_j)}^{p_j}\bigg)^{1/p_j}\\
&\lesssim \prod_{j=1}^{m}\Vert f_j\Vert_{L^{p_j}(w_j)},
\end{align*}
which proves \eqref{Lorentzinred}.

The remaining one follows immediately from the pointwise estimate \eqref{outpointwise} and Lemma \ref{prop:multiweighthardyest}.
Indeed,
\begin{align*}
\big\Vert T_{\mathrm{out}}(f_1,\dots,f_m)\big\Vert_{L^{p,\infty}(\nu)} &\lesssim \bigg\Vert \prod_{j=1}^{m}\HH \big(|f_j|\big)  \bigg\Vert_{L^{p,\infty}(\nu)}\lesssim \prod_{j=1}^{m}\Vert f_j\Vert_{L^{p_j}(w_j)}. 
\end{align*}
This completes the proof. \qed

\hfill

\section{Proof of Theorem \ref{thm:bilinear-local}}\label{sec:11}

Fix $R>0$. As in the proof of Theorem \ref{mainlinear1}, for $x\in B_R$ we write
$$S_T^{R,r}\big(f_1,\dots,f_m\big)(x)=\UU_R^r\big(T(f_1,\dots,f_m)\big)(x)$$
where
$$\UU_R^r h(x):=\bigg( \frac1{|x|^n}\int_{|y|\le |x|, |y|<R} \big|h(y)\big|^r\,dy\bigg)^{1/r}.$$

Now we write
\begin{equation*}
S_T^{R,r}\big(f_1,\dots,f_m\big)(x)\lesssim_r S_{T}^{R,r,\mathrm{in}}\big(f_1,\dots,f_m\big)(x)+S_{T}^{R,r,\mathrm{out}}\big(f_1,\dots,f_m\big)(x)
\end{equation*}
where
\begin{align*}
&S_{T}^{R,r,\mathrm{in}}\big(f_1,\dots,f_m\big)(x)\\
&:= \bigg( \frac1{|x|^n}\sum_{k\le 1+\log_2 R} \int_{\substack{|y|\le |x|\\ y\in I_k\cap B_R}} \big| T\big(\chi_{I_k^*}f_1,\dots,\chi_{I_k^*}f_m\big)(y)\big|^r
\,dy\bigg)^{1/r}
\end{align*}
and
\begin{align*}
&S_{T}^{R,r,\mathrm{out}}\big(f_1,\dots,f_m\big)(x)\\
&:=\bigg(\frac1{|x|^n}\sum_{k\le 1+\log_2 R} \int_{\substack{|y|\le |x|\\ y\in I_k\cap B_R}} \big| T\big(f_1,\dots,f_m\big)(y)-T\big(\chi_{I_k^*}f_1,\dots,\chi_{I_k^*}f_m\big)(y)\big|^r \,dy\bigg)^{1/r}.
\end{align*}
We estimate these two terms separately.

\subsection{Estimate of $S_{T}^{R,r,\mathrm{in}}$}
By H\"older's inequality with $p/r>1$,
\begin{align*}
&S_{T}^{R,r,\mathrm{in}}(f_1,\dots,f_m)(x)\\
&\lesssim \biggl( \frac1{|x|^n} \sum_{k\le 1+\log_2 R} \int_{\substack{|y|\le |x|\\ y\in I_k\cap B_R}} \big|T\big(\chi_{I_k^*}f_1,\dots,\chi_{I_k^*}f_m\big)(y)\big|^p \,dy \bigg)^{1/p}.
\end{align*}
Therefore, by Fubini's theorem,
\begin{align*}
& \big\Vert \chi_{B_R}S_{T}^{R,r,\mathrm{in}}\big(f_1,\dots,f_m\big)\big\Vert_{L^p(\nu)} \\
&\lesssim \bigg(\sum_{k\le 1+\log_2 R}\int_{I_k\cap B_R} \big|T\big(\chi_{I_k^*}f_1,\dots,\chi_{I_k^*}f_m\big)(y)\big|^p \Gamma_{R}\nu(y)\,dy\bigg)^{1/p},
\end{align*}
where $\Gamma_R$ is defined as in \eqref{gammatdef}.
Then H\"older's inequality yields
\begin{equation*}
\big( \Gamma_{R}\nu(y) \big)^{1/p} \le \prod_{j=1}^{m}\big( \Gamma_{R}w_j(y)\big)^{1/p_j}.
\end{equation*}
Moreover, in view of \eqref{supgammaest}, for each $j=1,\dots,m$,
\begin{equation*}
\esssup_{I_k^*} \Gamma_{R}w_j \lesssim \essinf_{I_k^*} \Big(M w_j+\Gamma_{R}w_j\Big).
\end{equation*}
Hence
\begin{equation}\label{esssupbilinearga}
 \Big(\esssup_{ I_k}\Gamma_{R}\nu \Big)^{1/p} \lesssim \prod_{j=1}^m \essinf_{I_k^*}\Big(Mw_j+\Gamma_{R}w_j\Big)^{1/p_j}. 
 \end{equation}
Using now the boundedness \eqref{bilinearLpbound} and \eqref{esssupbilinearga}, we obtain
\begin{align*}
&\bigg( \int_{I_k\cap B_R}  \big|T\big(\chi_{I_k^*}f_1,\dots, \chi_{I_k^*}f_m\big)(y)\big|^p \Gamma_{R}\nu(y)\,dy \bigg)^{1/p} \\
&\le \Bigl(\esssup_{ I_k} \Gamma_{R}\nu \Big)^{1/p} \big\Vert T\big(\chi_{I_k^*}f_1,\dots,\chi_{I_k^*}f_m\big)\big\Vert_{L^p(\bbrn)}\\
&\lesssim \prod_{j=1}^m \bigg(\essinf_{I_k^*}\Big(Mw_j+\Gamma_{R}w_j\Big)^{1/p_j} \big\Vert \chi_{I_k^*}f_j\big\Vert_{L^{p_j}(\bbrn)}\bigg)\\
&\le  \prod_{j=1}^{m}\mathcal{A}_{j,k}
\end{align*}
where
$$\mathcal{A}_{j,k}:=\bigg( \int_{I_k^*}|f_j(y)|^{p_j}\Big(Mw_j(y)+\Gamma_{R}w_j(y)\Big)\,dy\bigg)^{1/p_j}, \quad j=1,\dots,m.$$
Summing in $k\in\bbz$ and using H\"older's inequality for series, we arrive at
$$ \big\Vert \chi_{B_R} S_{T}^{R,r,\mathrm{in}}(f_1,\dots,f_m)\big\Vert_{L^p(\nu)} \lesssim \prod_{j=1}^{m}\bigg(\sum_{k\in\bbz} \mathcal{A}_{j,k}^{p_j}\bigg)^{1/p_j}.$$
Since $I_k^*\subset B_{4R}$ whenever $k\le 1+\log_2 R$, and $\Gamma_Rw_j(y)\le \Gamma_{4R}w_j(y)$,
the last expression is bounded by
$$ \prod_{j=1}^m \biggl( \int_{B_{4R}} \big|f_j(y)\big|^{p_j}Mw_j(y)\,dy + \int_{B_{4R}}\big|f_j(y)\big|^{p_j}\Gamma_{4R}w_j(y)\,dy \bigg)^{1/p_j}.$$
Since $w_j\in A_1$, using \eqref{a1character}, for each $j=1,\dots,m$,
$$ \int_{B_{4R}} |f_j(y)|^{p_j} Mw_j(y)\,dy \lesssim \Vert f_j\Vert_{L^{p_j}(w_j)}^{p_j} \lesssim \Vert f_j\Vert_{L^{p_j}\log L(w_j)}^{p_j} \quad \text{ uniformly in }R.$$
Also, by Fubini's theorem,
$$\int_{B_{4R}} |f_j(y)|^{p_j}\Gamma_{4R}w_j(y)\,dy = \int_{B_{4R}} \UU_{4R}\big( \chi_{B_{4R}}|f_j|^{p_j} \big)(x) w_j(x)\,dx.$$
Since $\UU_{4R}\big(\chi_{B_{4R}}|f_j|^{p_j}\big)(x) \lesssim M\big(|f_j|^{p_j}\big)(x),$
by Lemma \ref{utfl1uest}, the above term is controlled by a constant times 
$$ \max\big\{1,w_j(B_{4R})\big\} \big\Vert |f_j|^{p_j} \big\Vert_{L\log L(w_j)}\sim \max\big\{1,w_j(B_{4R})\big\} \big\Vert f_j \big\Vert_{L^{p_j}\log L(w_j)}^{p_j}.$$
Note that $w_j\in A_1$ implies that $w_j$ induces a doubling measure, and thus
$$w_j(B_{4R})\lesssim_{w_j} w_j(B_R) \quad \text{uniformly in $R$.}$$
Finally, we conclude that
$$ \Big\Vert \chi_{B_R}S_{T}^{R,r,\mathrm{in}}(f_1,\dots,f_m)\Big\Vert_{L^p(\nu)} \lesssim
\prod_{j=1}^m \Big( \max\big\{1,w_j(B_R)^{1/p_j}\big\} \Vert f_j\Vert_{L^{p_j}\log L(w_j)} \Big).$$

\medskip

\subsection{Estimate of $S_{T}^{R,r,\mathrm{out}}$.}
If $y\in I_k$ and  at least one of $z_j$, $j=1,\dots,m$, does not belong to $I_k^*$, then
$$ \sum_{j=1}^{m}|y-z_j| \sim \max\big\{|y|,|z_1|,|z_2|,\dots,|z_m|\big\}.$$
Since at least one of the variables $z_j$ belongs to $(I_k^*)^c$, the size condition gives
\begin{align*}
&\big|T(f_1,\dots,f_m)(y)-T(\chi_{I_k^*}f_1,\dots,\chi_{I_k^*}f_m)(y)\big|\\
&\lesssim \int_{(\mathbb R^n)^m} \frac{\prod_{j=1}^m |f_j(z_j)|} {\max\{|y|,|z_1|,\dots,|z_m|\}^{mn}} \,dz_1\cdots dz_m.
\end{align*}
The last expression is bounded by the product of $\HH\big(|f_j|\big)(y)$ over $j=1,\dots,m$, because
 $$ \frac{1}{\max\{|y|,|z_1|,|z_2|,\dots,|z_m|\}^{mn}} \le \prod_{j=1}^{m}\min\Big\{ \frac{1}{|y|^n},\frac1{|z_j|^n}\Big\}.$$
Hence
\begin{align*}
 S_{T}^{R,r,\mathrm{out}}\big(f_1,\dots,f_m\big)(x) &\lesssim \bigg( \frac1{|x|^n}\int_{|y|\le |x|} \chi_{B_R}(y)\Big(\prod_{j=1}^{m}\mathcal{H}\big(|f_j|\big)(y)\Big)^{r} \,dy \bigg)^{1/r}\\
 & \lesssim \bigg( M\bigg(\prod_{j=1}^{m} \HH\big(|f_j|\big)^r\bigg)(x)\bigg)^{1/r}.
 \end{align*}
Since $w_1,\dots,w_m\in A_1$, it follows that $\nu=\prod_{j=1}^{m}w_j^{p/p_j}\in A_1$.
Finally, using the $L^{{p}/{r}}(\nu)$ boundedness of the maximal operator $M$ in \eqref{weightHLmax}, H\"older's inequality, and \eqref{hhlpwest}, we obtain
\begin{align*}
\big\Vert \chi_{B_R} S_{T}^{R,r,\mathrm{out}}(f_1,\dots,f_m)\big\Vert_{L^p(\nu)} &\lesssim  \prod_{j=1}^{m}\big\Vert \HH(|f_j|)\big\Vert_{L^{p_j}(w_j)} \lesssim \prod_{j=1}^{m}\Vert f_j\Vert_{L^{p_j}(w_j)}.
\end{align*}
The desired result follows from the embedding $L^{p_j}\log{L}(w_j)\hookrightarrow L^{p_j}(w_j).$ \qed

\hfill

\section{Proof of Theorem \ref{thm:bilinear-global}}\label{sec:12}

Let $u_j\in A_1$, $j=1\dots,m$, and set
$$\varrho:=\prod_{j=1}^{m}u_j^{p/p_j}.$$
We first claim the estimate
\begin{equation}\label{eq:bilinear-global-A1}
\big\Vert S_T^r(f_1,\dots,f_m)\big\Vert_{L^p(\varrho)}\lesssim \prod_{j=1}^m \Big\Vert \big(\HH(|f_j|^{p_j})\big)^{1/p_j}\Big\Vert_{L^{p_j}(u_j)} .
\end{equation}
To prove the claim, we fix $0<\alpha<n$, and for $N\in\mathbb N$, define
$$u_{j,N}(x):=\min\big\{u_j(x),N|x|^{-\alpha}\big\}, \qquad j=1,\dots,m.$$
Then for each $j=1,\dots,m$, $u_{j,N}\nearrow u_j$ almost everywhere as $N\to\infty$. Moreover, as in \eqref{una1ua1uniform} and \eqref{hhunxfinite}, each $u_{j,N}$ belongs to $A_1$ uniformly in $N$, and
$$\HH u_{j,N}(x)<\infty \qquad \text{for a.e. }x\in\bbrn.$$
Hence, by Lemma \ref{ca1}, for each $j=1,\dots, m$,
$$\HH u_{j,N}\in A_1 \qquad\text{and}\qquad [\HH u_{j,N}]_{A_1}\lesssim [u_j]_{A_1},$$
uniformly in $N$. 
In particular, by Lemma \ref{rdrd}, the weights $\HH u_{j,N}$ satisfy the dyadic annular condition \eqref{biconsondyad} uniformly in $N$.

We also set
$$\varrho_N:=\prod_{j=1}^{m}u_{j,N}^{p/p_j}$$
so that $\varrho_{N}\nearrow \varrho$ almost everywhere as $N\to\infty$.
Since $r<p$, by H\"older's inequality we have
$$ S_T^r\big(f_1,\dots,f_m\big)(x)\le \bigg(\frac1{|x|^n}\int_{|y|\le |x|}\big|T\big(f_1,\dots,f_m\big)(y)\big|^p\,dy\bigg)^{1/p}.$$
Therefore, by Fubini's theorem,
\begin{align*}
&\int_{\bbrn}\big|S_T^r\big(f_1,\dots,f_m\big)(x)\big|^p\varrho_N(x)\,dx \\
&\le \int_{\bbrn} \bigg(\frac{1}{|x|^n}\int_{|y|\le |x|} \big|T\big(f_1,\dots,f_m\big)(y)\big|^p\,dy\bigg)\varrho_N(x)\,dx  \\
&= \int_{\bbrn} \big|T\big(f_1,\dots,f_m\big)(y)\big|^p \bigg(\int_{|x|\ge |y|}\frac{\varrho_N(x)}{|x|^n}\,dx\bigg)\,dy  \\
&\le \int_{\bbrn} \big|T\big(f_1,\dots,f_m\big)(y)\big|^p \HH\varrho_N(y)\,dy .
\end{align*}
By H\"older's inequality applied to the positive kernel defining $\HH$, we obtain
$$ \HH\varrho_N(y)  \le \prod_{j=1}^{m}\big(\HH u_{j,N}(y)\big)^{p/p_j}.$$
Consequently,
\begin{align*}
&\int_{\bbrn}\big|S_T^r\big(f_1,\dots,f_m\big)(x)\big|^p\varrho_N(x)\,dx \\
&\le \int_{\bbrn} \big|T\big(f_1,\dots,f_m\big)(y)\big|^p \prod_{j=1}^{m}\big(\HH u_{j,N}(y)\big)^{p/p_j}\,dy .
\end{align*}
Since each $\HH u_{j,N}$ belongs to $A_{p_j}$ and satisfies \eqref{biconsondyad}, we may apply Theorem \ref{bilinearHerz} in the setting $q_j=p_j$, $j=1,\dots,m$.  This gives
\begin{align*}
&\bigg( \int_{\bbrn} \big|T\big(f_1,\dots,f_m\big)(y)\big|^p \prod_{j=1}^{m}\big(\HH u_{j,N}(y)\big)^{p/p_j}\,dy \bigg)^{1/p}  \\
&\lesssim \prod_{j=1}^m \bigg(\int_{\bbrn}|f_j(y)|^{p_j}\HH u_{j,N}(y)\,dy\bigg)^{1/p_j}
\end{align*}
uniformly in $N$.
Using the symmetry of the kernel of $\HH$ and Fubini's theorem, we have
$$\int_{\bbrn} |f_j(y)|^{p_j}\HH u_{j,N}(y)\,dy = \int_{\bbrn} u_{j,N}(y) \HH\big(|f_j|^{p_j}\big)(y)\,dy \le \int_{\bbrn} u_j(y)\HH\big(|f_j|^{p_j}\big)(y)\,dy.$$
Thus
$$\big\Vert S_T^r(f_1,\dots,f_m)\big\Vert_{L^p(\varrho_N)} \lesssim \prod_{j=1}^m \Big\Vert \HH\big(|f_j|^{p_j}\big)^{1/p_j} \Big\Vert_{L^{p_j}(u_j)},$$
with a constant independent of $N$. Letting $N\to\infty$ and using the monotone convergence theorem, we obtain \eqref{eq:bilinear-global-A1}.\\

We now apply Lemma \ref{multilinearRdf} to the tuple
$$ \Big(S_T^r(f_1,\dots,f_m), \HH\big(|f_1|^{p_1}\big)^{1/p_1},\cdots, \HH\big(|f_m|^{p_m}\big)^{1/p_m} \Big).$$
From \eqref{eq:bilinear-global-A1}, it follows that for every $q_j>p_j$, $j=1,\dots,m$, with $1/q_1+\cdots+1/q_m=1/q$, and for every $w_j\in A_{q_j/p_j}$, $j=1,\dots,m$, we have
$$\big\Vert S_T^r(f_1,\dots,f_m)\big\Vert_{L^q(\nu)} \lesssim \prod_{j=1}^m \Big\Vert \HH\big(|f_j|^{p_j}\big)^{1/p_j} \Big\Vert_{L^{q_j}(w_j)}= \prod_{j=1}^m
\big\Vert \HH\big(|f_j|^{p_j}\big)\big\Vert_{L^{q_j/p_j}(w_j)}^{1/p_j},$$
where $\nu:=\prod_{j=1}^{m}w_j^{q/q_j}$.
Since $q_j/p_j>1$ and $w_j\in A_{q_j/p_j}$, the weighted boundedness of $\HH$,
namely \eqref{hhlpwest}, yields
$$\big\Vert \HH\big(|f_j|^{p_j}\big)\big\Vert_{L^{q_j/p_j}(w_j)} \lesssim \big\Vert |f_j|^{p_j}\big\Vert_{L^{q_j/p_j}(w_j)} = \Vert f_j\Vert_{L^{q_j}(w_j)}^{p_j}, \qquad j=1,\dots,m.$$
Finally,
$$\big\Vert S_T^r(f_1,\dots,f_m)\big\Vert_{L^q(\nu)} \lesssim \prod_{j=1}^{m}\Vert f_j\Vert_{L^{q_j}(w_j)}.$$
Since
$$ \big\Vert T(f_1,\dots,f_m)\big\Vert_{\mathcal{C}_q^{(r)}(\nu)} = \big\Vert S_T^r(f_1,\dots,f_m)\big\Vert_{L^q(\nu)},$$
the proof is complete.  \qed

\hfill

\section{Proof of Theorem \ref{thm:bilinear-global2}}\label{sec:13}

Let $1<q_1,\dots,q_m<\infty$ and let $w_j\in A_{q_j}$.
By the openness property of Muckenhoupt weights, for each $j$ there exists $1<t_j<q_j$ such that $w_j\in A_{q_j/t_j}$.
Setting $1/t=1/t_1+\cdots+1/t_m$, we have $r<t<q$.
 By the hypothesis, $T$ is bounded from $L^{t_1}\times\cdots\times L^{t_m}$ into $L^t$.
Applying Theorem \ref{thm:bilinear-global} with $p_j=t_j$ gives the desired result. \qed

\hfill


\begin{thebibliography}{99}


\bibitem{Al_Gr2026}
A. Al-Salman and L. Grafakos, \emph{On rough oscillatory singular integral operators}, Forum Math. \textbf{38} (2026), 955--975.

\bibitem{Al_Pa2002}
A. Al-Salman and Y. Pan, \emph{Singular integrals with rough kernels in $L\log{L}(\mathbb{S}^{n-1})$}, J. Lond. Math. Soc. (2) \textbf{66} (2002), no. 1, 153--174.

\bibitem{am:am} 
S.V. Astashkin  and L. Maligranda,  \emph{Ces\`aro function spaces fail the fixed point property},  Proc. Amer. Math. Soc. \textbf{136}  (2008), no. 12, 4289--4294.

\bibitem{am2:am2} 
S.V. Astashkin  and L. Maligranda,  \emph{A short proof of some recent results related to Ces\`aro function spaces},  Indag. Math.  \textbf{24}  (2014), no. 3, 589--592.

\bibitem{bs:bs} 
A. Baernstein  and E.T. Sawyer,  \emph{Embedding and multiplier theorems for $H^p(\mathbb R^n)$},  Mem. Amer. Math. Soc.  \textbf{53} (1985), no. 318, iv+82 pp.


\bibitem{Be_Sh1988}
C. Bennett and R. Sharpley, \emph{Interpolation of Operators}, Academic Press, 1988.

\bibitem{Ca_Zy1952}
A.P. Calder\'on and A. Zygmund,  \emph{On the existence of certain singular integrals}, Acta Math. \textbf{88} (1952), 85-139.

\bibitem{Ca_Zy1956}
A.P. Calder\'on and A. Zygmund,  \emph{On singular integrals}, Amer. J. Math. \textbf{78} (1956), 289-309.

\bibitem{Ch_Ta2022}
Y. Chen and W. Tao, \emph{Quantitative weighted estimates for oscillatory singular integrals with rough kernels}, Bull. Korean Math. Soc. \textbf{59} (2022), 191--202.

\bibitem{Ch_Sh_Sh2025}
S.S. Choudhary, S. Shrivastava and K. Shuin, \emph{Sparse bounds for maximal oscillatory rough singular integral operators}, Bull. Sci. Math. \textbf{201} (2025), Paper No. 103612, 20pp.

\bibitem{Ch_Gr1995}
M. Christ and L. Grafakos, \emph{Best constants for two nonconvolution inequalities}, Proc. Amer. Math. Soc. \textbf{123} (1995), 1687--1693.


\bibitem{Cr_Ma2018}
D. Cruz-Uribe and J.M. Martell, \emph{Limited range multilinear extrapolation with applications to the bilinear Hilbert transform}, Math. Ann. \textbf{371} (2018), 615--653.







 \bibitem{Do_Sl2024} 
 G. Dosidis and L. Slav\'ikov\'a, \emph{Multilinear singular integrals with homogeneous kernels near $L^1$},  Math. Ann. \textbf{389} (2024), 2259--2271. 
 
 \bibitem{Do_Sl_Park2026} 
 G. Dosidis, L. Slav\'ikov\'a, and B. Park, \emph{Bilinear rough singular integrals near the critical integrability via sharp Fourier multiplier criteria}, 
 Trans. Amer. Math. Soc. to appear. 

\bibitem{Duoa1993}
J. Duoandikoetxea,  \emph{Weighted norm inequalities for homogeneous singular integrals}, Trans. Amer. Math. Soc. \textbf{336} (1993), 869--880.

\bibitem{Duo2011} 
J. Duoandikoetxea, \emph{Extrapolation of weights revisited: new proofs and sharp bounds}, J. Funct. Anal. \textbf{260}  (2011),  no. 6, 1886--1901.


\bibitem{Duo_Ru1986} 
J. Duoandikoetxea and J.L. Rubio de Francia, \emph{Maximal and singular integral operators via Fourier transform estimates}, Invent. Math. \textbf{84}  (1986), 541--561.


\bibitem{f:f} 
W.G. Faris, \emph{Weak Lebesgue spaces and quantum mechanical binding},  Duke Math. J. \textbf{43}  (1976),  no. 2, 365--373.

\bibitem{Fe1979} 
R. Fefferman, \emph{A note on singular integrals}, Proc. Amer. Math. Soc. \textbf{74}  (1979),  no. 2, 266--270.

\bibitem{fw:fw} 
H.G. Feichtinger and F. Weisz, \emph{Herz spaces and summability of Fourier transforms},  Math. Nachr. \textbf{281}  (2008),  no. 3, 309--324.

\bibitem{gh:gh} 
J. Garc\'ia-Cuerva and M.J. Herrero, \emph{A theory of Hardy spaces associated to the Herz spaces}, Proc. London Math. Soc.  \textbf{69}  (1994),  no. 3, 605--628.

\bibitem{gr:gr}
J. Garc\'{\i}a-Cuerva and J.L. Rubio de Francia, \emph{Weighted Norm Inequalities and Related Topics}, North-Holland Mathematics Studies, Vol.~116, Notas de Matem\'atica [Mathematical Notes], 104. North-Holland Publishing Co., Amsterdam, 1985.

\bibitem{Gr_He_Ho2018}
L. Grafakos, D. He and P. Honz\'ik, \emph{Rough bilinear singular integrals}, Adv. Math. \textbf{326} (2018), 54--78.

\bibitem{Gr_He_Ho_Park_JLMS}
L. Grafakos, D. He, P. Honz\'ik, and B. Park,  \emph{Multilinear rough singular integral operators}, J. London Math. Soc. \textbf{109} (2024), e12867, 35pp.

\bibitem{gly:gly} L. Grafakos, X. Li and D. Yang, \emph{Bilinear operators on Herz-type Hardy spaces}, Trans. Amer. Math. Soc.  \textbf{350} (1998),  no. 3, 1249--1275.

\bibitem{Gr_To2002} 
L. Grafakos and R.H. Torres, \emph{Multilinear Calder\'on-Zygmund Theory}, Adv. Math.  \textbf{165} (2002),  124--164.





\bibitem{He_Park2023}
D. He and B. Park,  \emph{Improved estimates for bilinear rough singular integrals}, Math. Ann. \textbf{386} (2023), 1951-1978.

\bibitem{h:h} 
C.S. Herz, \emph{Lipschitz spaces and Bernstein's theorem on absolutely convergent Fourier transforms}, J. Math. Mech.  \textbf{18} (1968),  283--323.


\bibitem{JiangLu1995}
Y.S. Jiang and S.Z. Lu, \emph{Oscillatory singular integrals with rough kernel}, in \emph{Harmonic Analysis in China}, Math. Appl., vol. 327, Kluwer Academic Publishers, Dordrecht, 1995, 135--145.

\bibitem{kk:kk} 
N.J. Kalton and A. Kami\'nska, \emph{Type and order convexity of Marcinkiewicz and Lorentz spaces and applications}, Glasg. Math. J.  \textbf{47} (2005),  no. 1, 123--137.


\bibitem{lm:lm} 
K. Le\'snik and L. Maligranda , \emph{Abstract Ces\`aro spaces. Optimal range.}, Integral Equations Operator Theory   \textbf{81} (2015),  no. 2, 227--235.

 \bibitem{Lu_Zh1992} 
S.Z. Lu and Y. Zhang, \emph{Criterion on $L^p$-boundedness for a class of oscillatory singular integrals with rough kernels}, Rev. Mat. Iberoam.  \textbf{8} (1992),  no. 2, 201--219.

\bibitem{m:m} 
B. Muckenhoupt, \emph{Weighted norm inequalities for the Hardy maximal function},   Trans. Amer.  Math. Soc.  \textbf{165}  (1972), 207--226.


\bibitem{Na1986} 
J. Namazi, \emph{A singular integral}, Proc. Amer. Math. Soc.  \textbf{96}  (1986), no. 3, 421--424.

 \bibitem{Oj2000} 
H. Ojanen, \emph{Weighted estimates for rough oscillatory singular integrals}, J. Fourier Anal. Appl.  \textbf{6}  (2000), no. 4, 427--436.


\bibitem{Park_submitted} 
B. Park, \emph{Weighted estimates for multilinear singular integrals with rough kernels}, submitted.

 \bibitem{RicciStein1987}
F. Ricci and E.M. Stein, \emph{Harmonic analysis on nilpotent groups and singular integrals. I. Oscillatory integrals}, J. Funct. Anal. \textbf{73} (1987), no. 1, 179--194.


 \bibitem{So_We1994} 
 F. Soria and G. Weiss, \emph{A remark on singular integrals and power weights},  Indiana Univ. Math. J.  \textbf{43} (1994), no. 1, 187--204. 
 
\bibitem{St1967} 
E.M. Stein, \emph{Note on singular integrals}, Proc. Amer. Math. Soc.  \textbf{8} (1967),  250--254. 


  \bibitem{s:s} 
 J. Shiue, \emph{On the Ces\`aro sequence spaces},  Tamkang J. Math.  \textbf{1} (1970), no. 1, 19--25.
 
 \bibitem{Wa1990}
D. Watson, \emph{Weighted estimates for singular integrals via Fourier transform}, Duke Math. J. \textbf{60} (1990), 389--399.
 
 


\end{thebibliography}
\end{document}